\documentclass[11 pt, reqno]{amsart}
\usepackage{amsfonts}
\usepackage[left= 24mm, right= 24mm]{geometry}
\usepackage{amssymb, amsmath}
\usepackage{lscape}
\usepackage{dsfont}
\usepackage{mathrsfs}
\usepackage{mathtools}
\usepackage{bbm}
\usepackage{enumitem}

\usepackage{cancel}
\usepackage{centernot}
\usepackage{verbatim}
\usepackage{appendix}

\usepackage[backend=biber]{biblatex}
\DeclareMathAlphabet{\mathpzc}{OT1}{pzc}{m}{it}
\newtheorem{te}{Theorem}[section]
\newtheorem{defin}[te]{Definition}
\newtheorem{os}[te]{Remark}
\newtheorem{prop}[te]{Proposition}

\newtheorem{ex}[te]{Example}
\newtheorem{cor}[te]{Corollary}

\usepackage[pdftex,plainpages=false,colorlinks,hyperindex,bookmarksopen,linkcolor=red,citecolor=red,urlcolor=red]{hyperref}

\numberwithin{equation}{section}

\allowdisplaybreaks
\def \l { \left( }
\def \r {\right) }

\def \lq {\left[ }
\def \rq {\right] }

\def \lg {\left\{ }
\def \rg {\right\} }

\def \E {\mathbb{E}}

\def \N {\mathbb{N}}
\def \R {\mathbb{R}}

\def \vp {\varphi}
\def \ve {\varepsilon}
\def \si { \sigma }
\def \al { \alpha }

\usepackage{bm}

\newtheorem*{assumption*}{\assumptionnumber}
\providecommand{\assumptionnumber}{}
\makeatletter

\newcounter{mylabelcounter}
\newcommand{\labelText}[2]{%
#1\refstepcounter{mylabelcounter}%
\immediate\write\@auxout{%
  \string\newlabel{#2}{{1}{\thepage}{{\unexpanded{#1}}}{mylabelcounter.\number\value{mylabelcounter}}{}}%
}%
}

\makeatother

\def \bfn {\eta}
\def \L {\mathcal{L}}

\def \Inv { \psi }
\def \Int { \mathcal{I} }

\begin{document}

\large

%\title[]{A class of non-Markovian  processes with long-range dependence}

%\title[]{On some Non-Markovian processes with long-range dependence}

%\title{Non-Markovian chains with long-range dependence and scaling limits of Continuous Time Random Walks}

\title{Non-Markovian chains with long-range dependence and their scaling limits}

\author[]{Lorenzo Facciaroni}
 \author[]{Costantino Ricciuti}
\author[] {Enrico Scalas}

\keywords{Non-Markovian dynamics, continuous-time random walks, non-local operators, power law tails, anomalous diffusion, stable processes}
	\date{\today}
    \subjclass[2020]{60G99, 60G09, 60F17, 60G52}

\begin{abstract}

There is a well-established theory linking certain semi-Markov chains and continuous-time random walks to time-fractional equations and anomalous diffusion. In this work, we go beyond the semi-Markov framework by considering some  non-Markovian chains,  which exhibit long-memory behaviour, due to stochastic dependence among their  waiting times. 
 Particular attention is devoted to the so-called para-Markov chains. Their waiting times share the same marginal distributions as those of the above mentioned semi-Markov chains, but they are dependent; their joint distribution is of Schur-constant type and is closely related to complete Bernstein functions and De Finetti's theorems.  A second model that we focus on is given by time-changed Markov chains, where the random time is the inverse of an increasing stable process. This generalizes well-known semi-Markov models available in the literature, which typically focus solely on the inverse of the Lévy stable subordinator. The above mentioned models are unified by a general theory of time change of Markov chains.

\end{abstract}

\maketitle
\tableofcontents

\section{Introduction and preliminaries} \label{Introduction}

Throughout the paper, we use $\N$ and $\R$ to denote the sets of natural and real numbers, respectively; $\N_+$ is for positive integers and $\R_+$ for non-negative reals.

\subsection{State of art and vision}

Non-Markovian processes are often difficult to analyse.
\begin{comment}Except for a few notable classes, such as Gaussian or stable processes, it is uncommon to encounter non-Markovian models with tractable finite-dimensional distributions.
\end{comment}
A common strategy  to circumvent difficulties is to construct them via time-changes of Markovian processes. As is done in several models known in literature (see the discussion below), one may consider a time-homogeneous Markov process $M = \{M(t),\ t\in\R_+\}$, and a right-continuous, a.s. increasing process $\sigma = \{\sigma(t),\, t \in\R_+\}$, independent of $M$, with no deterministic points of discontinuity; then, letting $\Inv = \{\Inv(t),\, t \in\R_+\}$ denote the inverse of $\si$, i.e.
    \begin{align*}
        \Inv(t) := \inf\lg s\in\R_+ \mid \sigma(s) > t\rg, \qquad t\in\R_+,
    \end{align*}
one can define the (possibly) non-Markovian process $Y = \{Y(t),\ t\in\R_+\}$ as
\begin{align}
Y(t):= M(\Inv(t)) \qquad t\in\R_+. \label{prima formula}
\end{align}
We are particularly interested in two cases. The first one is when $M$ is a Brownian motion and thus $Y$ is an anomalous diffusion process; the second one is when $M$ is a homogeneneous continuous-time Markov chain. In the latter case, starting from  $\Tilde{M} = \lg \Tilde{M}(n),\ n\in\mathbb{N} \rg$,  a discrete-time Markov chain, and $N = \lg N(t),\ t\in\R_+ \rg$, a Poisson process of intensity $\lambda>0$, we  then define $M$ by
    \begin{align}
        M(t) := \tilde{M}(N(t)), \qquad \forall t\in\R_+. \label{continuous-time chain M(t)}
    \end{align}
The waiting times $\{W_k,\ k\in \N_+\}$ of $M$ are i.i.d. exponential random variables of parameter $\lambda$. Then one can prove that the waiting times 
 $\{J_k,\ k\in \N_+\}$ of $Y$ are given by
 
\begin{comment}By the definition of the time-changed process, in this case one has
    \begin{align*}
        Y(t) = \tilde{M}\l{N(\Inv(t))}\r \qquad T_n^- \leq \Inv(t) < T_{n+1},
    \end{align*}
    where $T_n := \sum_{k = 1}^nW_k,\ \ T_0 := 0$, and thus the waiting times $\{J_k,\ k\in\N_+\}$ of $Y$ are given by
\end{comment}

    \begin{align}
        J_{n+1} = \sigma(T_{n+1}) - \sigma(T_n^-), \qquad \forall n\in\mathbb{N}. \label{Preliminaries: generica forma del n+1 waiting time dopo il time change}
    \end{align}
    where $T_n:= \sum _{k=1}^n W_k,\ T_0 := 0$ and meaning with the superscript on $T_n$ the left limit.
    By Equation (\ref{Preliminaries: generica forma del n+1 waiting time dopo il time change}) it is worth noting that the joint law of the waiting times only depends on the finite-dimensional distributions of the chosen $\si$, whence different scenarios may arise.
    
To outline the situation in a concise, yet effective, manner, we may say that three scenarios are of particular interest:
    \begin{enumerate} \label{Elenco puntato sui casi di sigma}
        \item If $\sigma$ is a Lévy subordinator, then $Y$ is a non-Markovian process with short range dependence, namely a time-homogeneous semi-Markov process (in the sense of \cite{gikhman2004theory}). This case has been widely studied in the literature (see, for instance, the references below) and we shall recall the major facts in the discussion that follows.
        \item If $\si$ is an additive subordinator, i.e.~it has independent but not (necessarily) stationary increments, the resulting process $Y$ is a time-inhomogeneous semi-Markov process. As we shall see below, this case has already been studied in a few papers.
        \item The case of $\si$ with dependent increments seems to be neglected in the literature. In this case $Y$ displays long-range dependence and goes beyond semi-Markov models. In order to restrict our attention to models that  are both statistically relevant and analytically tractable,  we shall mainly consider the case where $\sigma$ has dependent but stationary increments. 
    \end{enumerate}

\begin{comment}
We are specifically interested in the case  As a general fact, if $\si$ has stationary increments, by a conditioning argument on the values of $\Inv(t)$, for each marginal waiting time $J_k,\ k\in\N$, the following representation holds
    \begin{align}
        P(J_k > t) = P(N(\Inv(t)) = 0) = \mathbb{E} e^{-\lambda \Inv(t)}. \label{Preliminaries: rappresentazione sopravvivenza waiting time marginale}
    \end{align}
\end{comment}

Concerning case (1), significant results are available in both Analysis and Statistical Physics (consult, e.g.~\cite{baeumer2017fokker,hernandez2017, Meerschaert2014,meerschaert2008triangular, meerschaert2019relaxation} and all references therein).  
Some facts are relevant to our purposes.

From an analytical perspective, the semi-Markov process $Y$ provides a stochastic solution to suitable non-local, integro-differential equations, possibly exhibiting fractional operators. In particular, let $\sigma$ be a subordinator with Lévy measure $\mu$ and Lévy tail $\overline{\mu}$. Moreover, let $G$ be the generator of the Markov process $M$, hence $M$ is governed by a differential equation $\partial/ \partial _t\, p(x,t)= Gp(x,t)$. Then $M(\Inv(t))$ is governed by an integro-differential equation of the form 
\begin{align}\label{eq 1}\mathcal{D}_t ^{\mu}p(x,t)=Gp(x,t)
\end{align}
where, for a sufficiently regular $g(\cdot)$, $\mathcal{D}_t ^{\mu}$ is the generalized Caputo fractional derivative, defined by
$$\mathcal{D}^\mu_t g(t):= \int _0^\infty (g(t)-g(t-\tau)) \mu (d   \tau) - \overline{\mu}(t) g(0).$$
An important case is the one where $\si$ is an $\al$-stable subordinator. Here the Lévy measure of $\si$ has a  density $\mu(t) = \frac{\al t^{-\al - 1}}{\Gamma (1 - \al)}$ and the operator $\mathcal{D}_t^{\mu}$ becomes the Caputo fractional derivative $\l \partial / \partial t \r^{\al}$ (see e.g. \cite{mainardigorenflo}). 

In the case of $M$ continuous-time, homogeneous Markov chain, as in Equation \eqref{continuous-time chain M(t)}, the waiting times (\ref{Preliminaries: generica forma del n+1 waiting time dopo il time change}) are i.i.d., but not (necessarily) exponential anymore, i.e.~$Y$ is a homogeneous semi-Markov chain. Each waiting time has survival function  
\begin{align}P(J_k>t)= S_\lambda (t) \label{funz sopravv}
\end{align}
where $S_\lambda(\cdot)$ is an eigenfunction of $\mathcal{D}_t^{\mu}$ with eigenvalue $-\lambda$.
For example, in the  case where $\si$ is a stable subordinator, the waiting times have power-law decaying distribution, known as  Mittag-Leffler distribution, i.e.~\begin{align}\label{mittag leffler distribution}S_\lambda (t)= E_\alpha (-\lambda t^\alpha),\end{align} where $E_\alpha(\cdot)$ is the Mittag-Leffler function (see \cite{haubold2011mittag}). In this case, for large $t$ we have $S_\lambda (t) \sim C\, t^{-\alpha}$, with $C>0$.

Finally, if $M$ is a Brownian motion, then one gets  models of  semi-Markov anomalous (sub-) diffusion, which are relevant  from the physical point of view. These processes appear as scaling limits of continuous-time random walks (CTRWs) with i.i.d. waiting times between successive jumps (see \cite{straka2011lagging,Meerschaert2014}). In this setting,  the density $p(x,t)$ of the time changed process $Y$ is the unique solution of the  non-local equation in time
\begin{align*}
    \mathcal{D}^\mu_t p(x,t) = \frac{1}{2}\Delta p(x,t), 
\end{align*}
where $\Delta$ the Laplacian operator.
%and $\delta(x)$ indicates the Dirac delta in $x$.
As before, the case of $\si$ $\al$-stable subordinator leads to the time fractional heat equation
\begin{align} \label{kkk}
    \l \frac{\partial}{\partial t}\r^{\alpha}  p(x,t) = \frac{1}{2}\Delta p(x,t).
\end{align}

Case (2) concerns the situation in which $\sigma$ is an additive process (see \cite{ken1999levy} for details). In this setting, $\sigma$ has independent but non-stationary increments, that is, it is characterized by a time-dependent Lévy measure $\mu_t(\cdot)$. For example, if $\sigma$ is a  multistable subordinator, then the  Lévy measure is $\mu _t(dx)=  \frac{\alpha (t)x^{-\alpha (t)-1}}{\Gamma (1-\alpha (t))}dx$, for an assigned function $t\longmapsto \alpha (t) $ with values in $(0,1)$. Some results are available regarding the time-changed process $Y$ which we summarize as follows. If $M$ is a continuous-time Markov chain, the waiting times in \eqref{Preliminaries: generica forma del n+1 waiting time dopo il time change} are independent but non-identically distributed (for their marginal distributions see \cite{beghin2019time}).
 Moreover, there is an interesting link with the theory of  scaling limit of CTRWs, where the waiting times between successive jumps are independent but not identically distributed (see \cite{molchanov2015multifractional}).

To the best of our knowledge,  there are not many contributions in literature concerning case (3). This appears to be challenging, as the time-change $M(\psi (t))$ yields processes that are fully non-Markovian, exhibiting genuine long-range dependence. As mentioned before, two general models are of interest.
  When $M$ is a continuous-time Markov chain, then the waiting times (\ref{Preliminaries: generica forma del n+1 waiting time dopo il time change}) are dependent and their joint law is related to the finite-dimensional distributions of $\sigma$.
Second, if $M$ is a Brownian motion, then $Y$ is a anomalous diffusion with long-range dependence, which should be the scaling limit of CTRWs with dependent waiting times.
%\end{itemize}

\subsection{Models and main results}

Our results concern non Markovian chains with dependent waiting times and anomalous diffusion with long range dependence.
Section \ref{Section: para makov chains} deals with the so-called para-Markov chains and para-Markov anomalous diffusion. Such processes were already  defined in \cite{facciaroni2025markov} and \cite{facciaroni2025random} in  particular cases;  a much more general definition is given in this paper and then this definition is embedded in the subordination scheme of Formula (\ref{prima formula}).

A para-Markov chain is such that its waiting times  have the same marginal survival function (\ref{funz sopravv}) of the above-mentioned semi-Markov chains, but here they are stochastically dependent. Their joint distribution is of Schur constant type (see e.g.~\cite{barlow1992finetti,caramellino1994dependence,caramellino1996wbf})  which makes them statistically and analytically tractable, i.e.
$$P(J_1>t_1, ..., J_n>t_n)= S_\lambda(t_1+...+t_n),$$
where $S_\lambda(\cdot)$ is a completely monotone function, related to complete Bernstein functions (for the theory of Bernstein functions see \cite{schilling2012bernstein}).
A very special case of these chains, with power-law, Mittag-Leffler distributed waiting times (as in Formula (\ref{mittag leffler distribution})), was already defined in \cite{facciaroni2025markov} and   non-local equations of the form $\partial^\alpha _t p(x,t)= -(-G)^\alpha p(x,t)$ arise. In the Markov case we get $S_\lambda (t)= e^{-\lambda t}$ and the waiting times are independent.

A para-Markov anomalous diffusion $\{Z_t,\ t\in\R_+\}$, with values in $\mathbb{R}$, is defined by the following finite-dimensional distributions
\begin{align}
\mathbb{E}e^{i\l \xi_1 Z(t_1)+ \dots + \xi_n Z(t_n) \r} = S _\lambda\l  \frac{1}{2} \langle \bm{\xi}, Q \bm{\xi} \rangle \r, \qquad  \bm{\xi}:=(\xi_1, \dots, \xi_n) \in \R^{n}, 
    \end{align}
    where the square matrix $Q$ is defined by
 $
        Q_{ij} =  t_i \wedge  t_j
   $. When $S_\lambda (t)= e^{-\lambda t}$ it reduces to  Brownian motion.
The special case where $S_\lambda(\cdot)$ has the form $(\ref{mittag leffler distribution})$ was already defined in \cite{facciaroni2025random},  with applications in kinetic theory; its governing equation has the form $\partial^\alpha _t p(x,t)= -2^{-\alpha}(-\Delta)^\alpha p(x,t)$.

Our effort is to embed such para-Markov chains and anomalous diffusion in the subordination scheme (\ref{prima formula}). We prove that a linear process $\sigma (t) =t / \mathcal{L}_\lambda$
%\begin{align}\label{time change lineare}\sigma (t) =t / \mathcal{L}_\lambda \end{align}
is needed, where $\mathcal{L}_\lambda$ is a suitable positive random variable. In this case $\sigma$ clearly has stationary but dependent increments. 
Thanks to this, two main results are proved:
\begin{itemize}
    \item A non-local equation arises, which generalizes Equation \eqref{eq 1} holding in the semi-Markov case only. Such equation reads (see Theorem \ref{Theorem: Governing equation equazione governante}),
    \begin{align*}
        \mathcal{D}_t ^{\mu, -G} u(x,t)= -\lambda u(x,t), \qquad u(x,0)= h(x),
    \end{align*}
     where $G$ is the generator of the underlying Markov process and, for a sufficiently regular $g(\cdot,\cdot)$, the operator on the left hand side is defined as
    \begin{align*}
    \mathcal{D}^{\mu, -G}_t g(x,t) := -\int _0^\infty  G \, \mu (-G\tau )\, (g(x,t)-g(x,t-\tau) ) d\tau + G\int _t^\infty  \mu (-G\tau) g(x,0) d\tau.
    \end{align*}
This is a generalization of the Caputo derivative $\mathcal{D}^\mu_t$, exhibiting an operator version of the Lévy measure.
\item 
    The para-Markov anomalous diffusion arises as a scaling limit of  para-Markov CTRWs (the latter being particular cases of para-Markov chains). Unlike what happens in the semi-Markov setting (see cases (1) and (2)), these CTRWs are such that  the waiting times are dependent. This is one of the main results of this paper.
\end{itemize}
We observe that the linear case $\sigma(t) = t/\L_{\lambda}$ is the most basic example of $\sigma$ with dependent and stationary increments. However, our approach allows to consider more general models, provided that the finite-dimensional distributions of $\sigma$ are known and analytically tractable.

Section \ref{Section: schema concentrico} is then devoted to the study of time-changes of chains, in a more general framework. Specifically, in Subsection \ref{Models  with dependent waiting times} we consider different types of time changes which include the linear time change as a special case, still inducing dependence between waiting times. 
Finally,
Subsection \ref{Section o subsection Stable stabili} is devoted to  the case where  $\sigma$ is an increasing stable process, as studied in \cite{samorodnitsky1994stable}, i.e.~an increasing process such that all the finite-dimensional distributions are multivariate stable. 

As far as we know, these models are novel. Indeed, so far, only the inverse of Lévy stable subordinator has been used for time-changing techniques. We shall devote particular attention  to the sub-case where $\sigma$ is self-similar or has stationary increments; here the non-Markovian chain may possibly exhibit dependent, power law Mittag-Leffler waiting times, generalizing the model presented in \cite{facciaroni2025markov} to distribution that are not of Schur constant type.
\vspace{1cm}

\section{Para-Markov chains and para-Markov anomalous diffusive processes} \label{Section: para makov chains}

Here we are interested in studying a particular class of non-Markovian processes with long-range dependence.
We start with the analysis of a remarkable counting process in Subsection \ref{The case of the counting process}, which turns out to be a clever generalization of the counting process defined in \cite{facciaroni2025markov}. It will be later used to construct the para-Markov chains in Definition \ref{Definizione definition di para markov}. Moreover, in Subsection \ref{The case of the anomalous diffusion} the counting process will be used to define a Continuous-Time Random Walk approximation of the para-Markov anomalous diffusive process, the latter introduced in the same Subsection.
One of the main results of the present Section is the governing equation for such processes, obtained in Subsection \ref{Subsection Governing equation} in a more general framework.

\subsection{The counting process and related chains} \label{The case of the counting process}

Following the semi-Markov theory recalled in  the Introduction, if $N = \{N(t),\ t\in\R_+\}$ is a Poisson process of intensity $\lambda > 0$ and $H$ is a (driftless) subordinator with inverse $L = \{L(t), \ t\in\R_+\}$, then $\{N(L(t)),\ t\in\R_+\}$ is a renewal counting process with i.i.d. waiting times. Each waiting time has survival function of the form
\begin{align} \label{funzione di sopravvivenza}
    S_\lambda(t) = P\l N(L(t)) = 0 \r =\mathbb{E} e^{-\lambda L(t)},
\end{align}
as a consequence of Equation
\eqref{Preliminaries: generica forma del n+1 waiting time dopo il time change}. Indeed, $H$ is Lévy and then has stationary increments. See \cite{meerschaert2011fractional} for details. Moreover, $S_\lambda(\cdot)$ is the eigenfunction of an integro-differential operator, known as generalized Caputo derivative, which, for a sufficiently regular function $g(\cdot)$,   has the form
\begin{align}
\mathcal{D}^{\mu}_t g(t) := \int _0^\infty (g(t)-g(t-\tau) )  \mu (d\tau ) -  \int _t^\infty g(0) \mu (d\tau), \label{operatore per semi markov}
\end{align}
%where $\overline{\mu} (t) := \int _t^\infty \mu (\tau)d\tau$
where $\mu$ is the Lévy measure of $H$.
Then it is true that
    \begin{align} \label{equazione autovalori}
     \mathcal{D}^{\mu}_t  S_\lambda(t) = -\lambda S_\lambda(t), \qquad S_\lambda(0)=1.
    \end{align}
A standard reference for this theory is \cite{meerschaert2019stochastic}.
It is worth recalling that the probability mass function $p(x,t) := P(N(L(t)) = x)$ of the time changed process is governed by the following non-local equation in time
\begin{align}
    \mathcal{D}_t^{\mu} \, p(x,t) = - \lambda (I - B) \, p(x,t), \label{equazione governante fractional poisson process}
\end{align}
where $I$ is the identity operator and $B$ the lag operator, i.e.~$B \, p(x,t) = p(x-1, t), \ x\in\N$ and $p(-1,t) = 0$ for each $t\in\R_+$. For details we suggest to consult \cite{mainardi2000fractional}.

In the following, we shall consider the case where the survival function $S_\lambda (\cdot)$ is completely monotone, i.e.
        \begin{align} \label{mistura di esponenziali}
            S_\lambda(t) = \int_0^{\infty} e^{-s t} \nu_{\lambda}(ds),
        \end{align}
        being $\nu_{\lambda}(\cdot)$ a probability distribution. We shall indicate with $\mathpzc{E}$ the expected value with respect to this distribution.
    Equation \eqref{mistura di esponenziali} defines a mixture of exponential distributions. In  Theorem \ref{Orsingher Ricciuti Toaldo} we shall recall  that Equation (\ref{mistura di esponenziali}) is valid if and only if the Laplace exponent of the subordinator $H$ is a so-called complete Bernstein function. 
Here we recall some basic facts on Bernstein functions; details and proofs can be found in \cite{schilling2012bernstein}.

    A function $f : (0, \infty) \longrightarrow \R$ is said to be a Bernstein function if $f \in C^\infty(0, \infty)$, \ $f(\bfn) > 0$ for each $ \bfn > 0$ and $ (-1)^{n-1} f^{(n)}\l \bfn \r \geq0 \text{ for all } n \in \mathbb{N}_+,\ \bfn > 0$. Each Bernstein function has a unique integral representation,  known as Lévy-Khintchine representation.
Precisely, a function $f(\cdot)$ is a Bernstein function if and only if there exist constants $a,b \in\R_+$ and a measure $\mu$ on $(0, \infty)$ satisfying
$
\int_0^\infty (1 \wedge s) \, \mu(ds) < \infty$
such that
\begin{align*}
f(\bfn) = a + b \bfn + \int_0^\infty \l1 - e^{-\bfn s}\r\, \mu(ds).
\end{align*}
It is known that the class of Bernstein function is in 1-1 correspondence with the class of subordinators.
Indeed each Bernstein function is the Laplace exponent of a subordinator, i.e.~to each Bernstein function $f(\cdot)$ there corresponds a subordinator $H$ such that $\mathbb{E}e^{-\eta H(t)}= e^{-tf(\eta)}$, and vice versa.
    A Bernstein function is said to be complete  if its Lévy measure $\mu$ has a completely monotone density.
    %Here the representation reads
%\begin{align*} f(\bfn) = a + b \bfn + \int_0^\infty \l 1-e^{-\bfn s} \r m(s) ds. \end{align*}
We further recall that if  $g:(0,\infty) \longrightarrow \R$ is a completely monotone function, then it is the Laplace transform of a unique measure $\nu$ on $[0,\infty)$, i.e.~for all $\bfn > 0$ $g(\bfn) = \int_0^{\infty} e^{-\bfn t} \nu(dt)$.
    %Conversely, whenever $\int_0^{\infty} e^{-\bfn t} \nu(dt)$ is finite for every $\bfn > 0$, the function $\bfn \longmapsto \int_0^{\infty} e^{-\bfn t} \nu(dt) $ is completely monotone.
We now recall a characterization of complete Bernstein functions, which has been proved in \cite{orsingher2018semi} (Lemma 7.1).
    \begin{te}\label{Orsingher Ricciuti Toaldo}
        Let $H$ be a strictly increasing subordinator with Laplace exponent $f$ and let $\{L(t),\ t\in\R_+\}$ be the inverse process of $H$. Then
        \begin{align*}
            t\mapsto \mathbb{E} e^{-\lambda L(t)} \quad \text{is completely monotone} \iff f \quad \text{is a complete Bernstein function}.
        \end{align*}
    \end{te}
    For a rich list of complete Bernstein function, the reader can consult \cite{schilling2012bernstein}. Here we report a well known example.

  \begin{ex} \label{Esempio fractional PP, con distribuzione Mittag-Leffler}
  If $H$ is the $\alpha$-stable subordinator, then $f(\bfn ) = \bfn^{\alpha},\ \bfn > 0,\ \alpha \in (0,1]$, which is a complete Bernstein function. The resulting renewal process is the so-called fractional Poisson process (see \cite{beghin2009fractional,mainardi2004fractional,meerschaert2011fractional} and also \cite{kataria2022generalized} for related models), whose waiting times follow a so-called Mittag-Leffler distribution, i.e.~they have survival function $S_\lambda(t)=\mathbb{E} e^{-\lambda L(t)} = E_{\alpha}\l -\lambda t^{\alpha} \r$, where 
    \begin{align}
        E_{\al}(z) := \sum_{k = 0}^{\infty} \frac{z^k}{\Gamma(1+\al k)}, \qquad z\in\mathbb{C}, \label{Definizione mittag leffler function funzione}
    \end{align}
    is the one-parameter Mittag-Leffler function.
    In this case, one has
    \begin{align*}
        E_{\alpha}\l -\lambda t^{\alpha} \r = \int_0^{\infty} e^{-s  t} \nu_{\lambda}(ds),
    \end{align*}
    where $\nu_{\lambda}(ds)$ is the Lamperti distribution. For $\alpha < 1$, it admits the following representation
    \begin{align*}
        \nu_{\lambda}(ds) = \frac{\sin(\pi \alpha)}{\pi}\frac{\lambda s^{\alpha - 1}}{s^{2\alpha} + 2\lambda s^{\alpha} \cos(\pi\alpha) + \lambda^2} ds.
    \end{align*}
    Details can be found in \cite{james2010lamperti}.
\end{ex}

%As we have recalled, any renewal counting process is built from a sequence of independent waiting times, each following the same non-negative distribution.
What is missing in the above semi-Markov framework is the possibility to model counting processes in which the waiting times are no longer independent; we shall consider this case, while still preserving  the same marginal distribution \eqref{funzione di sopravvivenza}. Such a construction is essential for non-Markovian phenomena, where memory effects induce dependence between successive waiting times. %\cite{Tejedor2010}. 
The goal here is precisely to define a counting process that maintains the renewal-like structure of identically distributed waiting times but relaxes the independence assumption. We shall assume a Schur constant dependence, the latter discussed in details in Subsection \ref{Section: schema concentrico}.
% in order to have an analytically tractable counting process. For a discussion of this dependence structure we refer to Subsection \ref{Section: schema concentrico}.

    \begin{defin} \label{Definition of exchangeable fractional PP}
Let $\{J_k,\ k\in\N_+\}$ be a sequence of non-negative  random variables, such that, for all $n\in \N_+$ we have
\begin{align}
        P[J_1 > t_1,\ \ldots,\ J_n > t_n] = S_\lambda\left( \sum_{k=1}^n t_k\right) \label{Definizione di new fractional: congiunta degli intertempi}
    \end{align}
 where $t_k\geq0,\ k \in\{1,\ldots,n\}$ and $S_\lambda(\cdot)$ has the form of Equation \eqref{mistura di esponenziali}.
Then the process $\mathcal{N} = \{\mathcal{N}(t),\ t\in\R_+\}$ defined by
    \begin{align*}
        \mathcal{N}(t) := n \qquad t \in \left[ \sum_{k=1}^n J_k, \sum_{k=1}^{n+1} J_k \right),\quad t\in\R_+,\ n\in\N,
    \end{align*}
    is said to be para-Markov counting process.
\end{defin}

The above counting process is a non-Markovian counterpart of the Poisson process and goes beyond even the fractional Poisson process and renewal counting processes, thanks to long-range dependence.
For $S_{\lambda}(t) = e^{-\lambda t}$ we obtain the Poisson process $N$ of parameter $\lambda$ as a particular case. Other parameters of the above process depend on the chosen $S_{\lambda}(\cdot)$, i.e.~on the measure $\nu_{\lambda}(\cdot)$.
The special case of $S_{\lambda}(t) = E_{\alpha}(-\lambda t^{\alpha})$ as in Example \ref{Esempio fractional PP, con distribuzione Mittag-Leffler} has been already considered in \cite{facciaroni2025markov}.

As we have recalled in the Introduction, a counting process can be used to construct continuous-time chains. Specifically, using the above non-Markovian counting process, the resulting chains are completely non-Markovian. Let us properly define such processes.

\begin{defin} \label{Definizione definition di para markov}
    Let $\tilde{M} = \lg\tilde{M}(n),\ n\in\N\rg$ be a discrete-time Markov chain. Let  $\mathcal{N}$ be the para-Markov counting process as in Definition \ref{Definition of exchangeable fractional PP}. Then the process $Y = \{Y(t),\ t\in\R_+\}$ where 
    \begin{align*}
        Y(t) := \tilde{M} (\mathcal{N}(t)), \qquad t\in\R_+,
    \end{align*}
    is said to be a para-Markov chain.
\end{defin}

Parameters of this process are the same of the underlying $\mathcal{N}$ as well as of the chain $\tilde{M}$.
Note that this construction includes a large class of models: there are several complete Bernstein functions, each of which can be used to model different types of dependence between waiting times. Please note that the name \textit{para-Markov chains} has already been used in \cite{facciaroni2025markov} for the special case of Mittag-Leffler waiting times.

We are interested in the distributional properties of such non-Markovian processes. By definition of para-Markov chains, it will suffice to study the distributional properties of the counting process. The following proposition makes it analytically tractable. The symbol $\overset{d}{=}$ means equality of finite-dimensional distributions.

\begin{prop} \label{Proposition: uguaglianza in distribuzione del processo di conteggio}
    Let us consider the para-Markov counting process $\mathcal{N}$ and a Poisson process $N$ of intensity 1. Then there exists a non-negative random variable $\L_{\lambda}$ such that 
    \begin{align*}
        \mathcal{N}(t) \overset{d}{=} N \l t \L_{\lambda} \r, \qquad t\in\R_+.
    \end{align*}
\end{prop}

\begin{proof}
    The survival function $S_{\lambda}(\cdot)$ is completely monotone, so that there exists a probability measure $\nu_{\lambda}(\cdot)$ such that
    \begin{align*}
        S_{\lambda}(t) = \int_0^{\infty}e^{-st} \nu_{\lambda}(ds), \qquad \forall t\in\R_+.
    \end{align*}
    Let $\L_{\lambda}$ be a random variable whose distribution is $\nu_{\lambda}(\cdot)$.
    Consider that
    %We call $\tilde{M} = \{\tilde{M}(t),\ t\in\R_+\}$ the embedded Markov chain, i.e.
    \begin{align*}
        N(t) = n, \qquad \sum_{k = 1}^n W_k \leq t < \sum_{k = 1}^{n+1} W_k, \qquad n\in\N_+,
    \end{align*}
    where $W_k$ are i.i.d. exponential of parameter $1$. Considering the random scaling  $t\longrightarrow t \L_{\lambda}$, one has
    \begin{align*}
        N(t \L_{\lambda}) = n, \qquad \sum_{k = 1}^n W_k \leq t \L_{\lambda} < \sum_{k = 1}^{n+1} W_k,
    \end{align*}
    so that the waiting times $\{J_k,\ k\in\N_+\}$ of $\{N(t \L_{\lambda}),\ t\in\R_+\}$ take the form $J_k = W_k/\L_{\lambda}$.
    Thus, to prove that $N(t \L_{\lambda})$ coincides with $\mathcal{N}(t)$ in the sense of finite-dimensional distributions,  it is sufficient to show that the sequence of waiting times  $\{J_k,\ k\in\N_+\}$ of $N(t \L_{\lambda})$ has  joint distribution given by Equation \eqref{Definizione di new fractional: congiunta degli intertempi}. This can be done using a conditioning argument on the values of $\L_{\lambda}$, as follows
    \begin{align*}
        P \lq \frac{W_1}{\L_{\lambda}} > t_1,\ \ldots,\ \frac{W_n}{\L_{\lambda}} > t_n \rq  =
        \int_0^{\infty}e^{-l \sum_{k=1}^n t_k}\nu_{\lambda}(dl) \\
         = \mathpzc{E}\left[e^{- \L_{\lambda} \sum_{k=1}^n  t_k }\right]
         = S_{\lambda} \left(\sum_{k=1}^n  t_k\right),
    \end{align*}
    where the last equality follows from Equation \eqref{mistura di esponenziali}.
    This completes the proof.
\end{proof}

We stress that the process under consideration cannot be trivially characterized as a mixture of Markovian processes. Indeed, its definition relies on Equation \eqref{Definizione di new fractional: congiunta degli intertempi}, and the resulting dynamics is genuinely non-Markovian due to the dependence structure among the waiting times. Proposition \ref{Proposition: uguaglianza in distribuzione del processo di conteggio} merely establishes an equality in distribution, which follows directly from the defining construction of the process.
\begin{comment}
We observe that, using the notation of Proposition \ref{Proposition: uguaglianza in distribuzione del processo di conteggio}, the Equation \eqref{mistura di esponenziali} can be written as
\begin{align}
    S_{\lambda}(t) = \mathpzc{E} e^{- t \L_{\lambda}}, \label{S(t) come valore atteso rispetto a L}
\end{align}
where $\L_{\lambda}$ has distribution $\nu_{\lambda}(\cdot)$.
\end{comment}

\begin{cor} \label{Corollario: uguaglianza in distribuzione per le catene}
    Let us consider the process $Y$ as in Definition \ref{Definizione definition di para markov} and the continuous-time Markov chain $M = \lg M(t),\ t\in\R_+\rg$, where $M(t) = \tilde{M}(N(t))$ and $N$ is the Poisson process of intensity $1$. Then there exists a non-negative random variable $\L_{\lambda}$ such that 
    \begin{align*}
        Y(t) \overset{d}{=}M(\L_{\lambda} t), \qquad t\in\R_+.
    \end{align*}
\end{cor}

\begin{proof}
    The result follows from the definition of $Y(t)$ and Proposition \ref{Proposition: uguaglianza in distribuzione del processo di conteggio}.
\end{proof}

\subsubsection{Heuristic derivation of governing equations}

Due to the dependence structure among waiting times \eqref{Definizione di new fractional: congiunta degli intertempi}, we expect a non-local equation in space-time variables for the probability mass function $p(x,t) := P(\mathcal{N}(t) = x)$. We now present some considerations in this direction, obtaining a heuristic result which will be made rigorous in Subsection \ref{Subsection Governing equation} for a larger class of processes.

Inspired by arguments of operational functional calculus, we can re-write  (\ref{equazione autovalori}) as
\begin{align} 
\label{eq intermedia}
f\biggl (\frac{d}{dt} \biggr ) S_\lambda(t)- \overline{\mu} (t)= -\lambda S_\lambda(t) 
\end{align}
where $f$ is the Bernstein function of the underlying subordinator and $\overline{\mu}$ is the tail of the Lévy measure.
Using a conditioning argument and Proposition \ref{Proposition: uguaglianza in distribuzione del processo di conteggio}, the Laplace transform of $\mathcal{N}(t)$ reads
    \begin{align*}
        \E e^{-\eta \mathcal{N}(t)} = \mathpzc{E} e^{- t \l1 - e^{-\eta} \r \L_{\lambda} }  =: A(\eta, t), \qquad \bfn > 0, \quad t\in\R_+.
    \end{align*}
    Because of the representation of $S_\lambda(t)$ in Equation \eqref{mistura di esponenziali},  it can be seen that $A(\eta, t) =S_\lambda \l t \l1 - e^{-\eta}\r \r$. Now, Equation (\ref{eq intermedia}) holds for each $t$, so also for $t(1-e ^{-\eta})$. Hence we get (remember that $A(\eta,0)=1$)
    \begin{align*}
        f\l \l 1 - e^{-\eta}\r^{-1}  \frac{\partial}{ \partial t} \r A(\eta, t) - \overline{\mu} (t(1-e^{-\eta}))A(\eta, 0)= -\lambda A(\eta, t).
    \end{align*}
We now need to return to the $x$ variable by Laplace inversion. Note that $\l 1 - e^{-\eta}\r^{-1}$ is the Laplace symbol of an operator that we identify as $(I-B)^{-1}$. Thus we heuristically  arrive at the following equation 
 \begin{align}
        f \l (I-B)^{-1} \frac{\partial }{\partial t} \r p(x, t) - \overline{\mu}\l t (I-B) \r p(x,0)= -\lambda  p(x, t), \qquad p(x,0) = \delta(x). \label{Equazione euristica del new fractional Poisson}
\end{align}    
In order to support the above intuition, we need to briefly discuss the meaning of $(I-B)^{-1}$ and the space of functions on which the operator will act.
We shall use the standard notation for the following Banach spaces $\l l^1, ||\cdot||_1 \r ,\ ||a||_1 := \sum_{k = 0}^{\infty} |a_k|  \ \text{for} \ a\in l^1$ and $\l l^\infty, ||\cdot||_\infty \r ,\ ||a||_\infty := \sup_{k \in\N} |a_k| \ \text{for} \ a\in l^\infty$.
\begin{comment}
\begin{align*}
    \l l^1, ||\cdot||_1 \r , &\qquad ||a||_1 := \sum_{i = 0}^{\infty} |a_i| \quad \text{for} \quad a\in l^1,  \\
    \l l^\infty, ||\cdot||_\infty \r , &\qquad ||a||_\infty := \sup_{i \geq 0} |a_i| \quad \text{for} \quad a\in l^\infty.
\end{align*}
\end{comment}
The operator is defined by

\begin{align*}
    (I-B)^{-1} : l^1 &\longrightarrow l^{\infty} \\
     \l a_n,\ n\in\N \r =: a &\longmapsto (I-B)^{-1} a := \l \sum_{k = 0}^n a_k,\ n\in\N \r.
\end{align*}
This operator coincides with the resolvent of $B$, i.e.~$R(z) := (I-zB)^{-1}$, evaluated in $z = 1$ (see Paragraph 3 in \cite{applebaum2009levy} for details). For this reason, we shall use the notation $(I-B)^{-1} =: R$. The operator $R$ is linear and continuous from $l^1$ to $l^{\infty}$, and hence bounded. Indeed, for each $a,b\in l^1$ and for each $\al,\beta \in \mathbb{C}$, one has
\begin{align*}
        R (\alpha a+\beta b)
        %&= \l \sum_{k = 0}^n (\alpha a_n + \beta b_n),\ n\in\N \r 
         &= \alpha \l \sum_{k = 0}^n a_n,\ n\in\N \r + \beta \l \sum_{k = 0}^n b_n,\ n\in\N \r
        = \al R a + \beta R b
    \end{align*}
as required. Moreover, for each $\varepsilon > 0$, one can choose $\delta = \varepsilon$ to get

\begin{align*}
    || a - b ||_1 = \sum_{i = 0}^{\infty} |a_i-b_i| <\delta \implies || R a - R b ||_{\infty} =  \sup_{n\in\N_+} \left| \sum_{k = 0}^n (a_k - b_k) \right| < \varepsilon.
\end{align*}

\begin{comment}
\begin{enumerate}
    \item Mostriamo che è continuo. Si ha
    \begin{align*}
        || R a - R b ||_{\infty} = \sup_{n\in\N_+} | \sum_{k = 0}^n (a_k - b_k) | < \varepsilon \impliedby | \sum_{k = 0}^n (a_k - b_k) | < \varepsilon \quad for \ each\ n\in\N.
    \end{align*}
    Scegliendo $\delta = \varepsilon$ si ha
    \begin{align*}
        || a - b ||_1 = \sum_{i = 0}^{\infty} |a-b_i| <\delta \implies || R a - R b ||_{\infty} < \varepsilon.
    \end{align*}
\end{enumerate}
\end{comment}

The following remark gives a probabilistic interpretation of $R$.

\begin{os}
    Let us consider a sequence of probability masses $p = (p(x, t),\ x\in\N)$ for a fixed $t\in\R_+$. Then
    \begin{align*}
        (I-B)^{-1} p(x,t) = \sum_{k = 0}^x p(k,t) = P(\mathcal{N}(t) \leq x).
    \end{align*}
    The operator turns out to have the property that for each $t$, if applied to $p(x,t)$, gives the c.d.f. of the counting process at time $t$ as output.
\end{os}

%Passing $\eta \longrightarrow x$ one heuristically obtains \eqref{equazione governante fractional poisson process} as expected.

\begin{os}
    An equation of type \eqref{Equazione euristica del new fractional Poisson} has been rigorously obtained in \cite{facciaroni2025markov} in  the special case $f(\eta) = \eta^\alpha$. Consistently with the above heuristic discussion, in this particular case the space and time operators separate so that Equation \eqref{Equazione euristica del new fractional Poisson} becomes
    \begin{align}
    \label{FCAA}
        \l\frac{\partial }{\partial t}\r^{\al} p(x,t) = -\lambda ^\alpha \l I-B\r^{\al} p(x,t),
    \end{align}
    being $(\partial / \partial t )^{\al}$ the Caputo fractional derivative. The process with distribution $p(x,t)$ that solves this equation has been named exchangeable fractional Poisson process. It is characterized by Schur constant dependence between waiting times, marginally distributed according to the Mittag-Leffler distribution. We emphasize that it is a special case of para-Markov counting process of Definition \ref{Definition of exchangeable fractional PP}. Indeed, the solution of \eqref{FCAA} is a particular case of the theory discussed here. 
\end{os}

%\textcolor{blue}{ES: Possiamo definire questo operatore usando la (3.3) attuale.}

\subsection{Governing equations} \label{Subsection Governing equation}
    Proposition \ref{Proposition: uguaglianza in distribuzione del processo di conteggio} and Corollary \ref{Corollario: uguaglianza in distribuzione per le catene} inspire the following problem.
    Let $M$ be a Markov  process. A well-established theory states that it is governed by a local equation of the form $(\partial/\partial t)\, u(x,t)= Gu(x,t)$, where  $G$ is the infinitesimal generator of $M$. Here we aim to find a governing equation for a non-Markovian process $Y = \{Y(t),\ t\in\R_+\}$, such that
    \begin{align}
        Y(t) \overset{d}{=} M(\mathcal{L}_\lambda t), \label{subordinato}
    \end{align}
    and $\mathcal{L}_\lambda$ is a random variable with distribution $\nu _\lambda$ defined in Equation (\ref{mistura di esponenziali}).
By non-Markovianity, we shall see that a non-local equation arises, exhibiting a pseudo-differential operator in time-space variables. 

In the semi-Markov case, the governing equation is well known. Indeed, as recalled in the Introduction, if $M$ is a Markov process and $L$ is the inverse of a subordinator $H$ independent of $M$, then the semi-Markov process $\{ M(L(t)),\ t\in\R_+\}$ is governed by
\begin{align} \label{semi markov}
   \mathcal{D}^{\mu}_t  u(x,t)= G u(x,t) 
\end{align}
where $G$ is the generator of $M$ and $\mathcal{D}_t^{\mu}$ is defined in \eqref{operatore per semi markov}. 

In Theorem \ref{Theorem: Governing equation equazione governante} we shall see that the governing equation of (\ref{subordinato})
is much more complicated given that the the spacial and the time operators are not, in general, de-coupled. We shall work in the particular case in which $M$ is a Lévy process.
To derive such governing equations, we need to define a new operator as follows
\begin{align}
\mathcal{D}^{\mu, -G}_t w(x,t) := -\int _0^\infty  G \, \mu (-G\tau )\, (w(x,t)-w(x,t-\tau) ) d\tau + G\int _t^\infty  \mu (-G\tau) w(x,0) d\tau.
\label{Para Markov non-local operator}
\end{align}
acting on a sufficiently regular function $w$.
\begin{comment}observe that, if $t\longmapsto \mu (t)$ is a Lévy density, then, for any $a>0$, the function $t\longmapsto a\mu (at)$ is also a Lévy density; its  associated Generalized Caputo derivative has the form
\begin{align*}
\mathcal{D}^{\mu, a}_t h(t)= \int _0^\infty (h(t)-h(t-\tau) ) a \mu (a\tau ) d\tau - h(0)\int _t^\infty a\mu (a\tau) d\tau
\end{align*}
with the notation $\mathcal{D}^{\mu, 1} = \mathcal{D}^{\mu} $.
\end{comment}
The term $G\mu (-G\tau)$ is meant in the following sense: since, by assumption, the Lévy measure has a completely monotone density, say $\mu (\cdot)$, then there exist a unique measure $\kappa$ such that
$$\mu (t)= \int _0^\infty e^{-tz} \kappa (dz)$$
and thus 
\begin{align}
    G\mu (-G\tau) w(x,t):=G \int _0^\infty \mathcal{T}_{\tau z} w(x,t) \kappa(dz) \label{G mu(-Gtau) parte dell'operatore}
\end{align}
where $\lg \mathcal{T}_t,\ t\in\R_+\rg$ is the semigroup generated by $G$.
\begin{os}
The operator \eqref{Para Markov non-local operator} can be heuristically constructed by means of $\mathcal{D}_t^\mu$ in formula \eqref{operatore per semi markov} in the following way.  Consider that the Caputo operator $\mathcal{D}_t ^\mu$ is associated to the Lévy density $\mu (t)$. For each $a>0$, the function $t\longmapsto a\mu (at)$ is also a Lévy density, and the associated Caputo operator is defined by
\begin{align}
\mathcal{D}^{\mu, a}_t h(t):= \int _0^\infty a \mu (a\tau ) (h(t)-h(t-\tau) )  d\tau - h(0)\int _t^\infty a\mu (a\tau) d\tau, \label{non lo so}
\end{align}
with the notation $\mathcal{D}^{\mu, 1} =: \mathcal{D}^{\mu} $. Being $G$ a spectrally negative operator, then one can substitute  $a$ with $-G$ and obtain the new definition.
\end{os}

We are now ready to state the theorem. We shall indicate with $\langle\cdot,\cdot\rangle$ the usual inner product.
We shall use $\mathcal{S}\l\mathbb{R}^n\r$ for the Schwartz space of functions on $\R^n$.
For  a function $h\in \mathcal{S}\l\mathbb{R}^n\r $, we shall use the following notation 
\begin{align*}
    \hat{h}\l\bm{\xi}\r &:= \int_{\R^n}e^{- i \langle \bm{\xi}, \bm{x} \rangle} h(\bm{x}) d\bm{x}, \quad \bm{\xi}\in \R^n, \\ 
    h\l\bm{x}\r &:= \frac{1}{(2\pi)^n}\int_{\R^n}e^{i \langle \bm{\xi}, \bm{x} \rangle} \hat{h}(\bm{\xi}) d\bm{\xi}, \quad \bm{x}\in \R^n.
\end{align*}
for the Fourier transform of $h$ and the inverse Fourier transform of $\hat{h}$ in the space variable.
Let $\L_{\lambda}$ be a random variable whose distribution $\nu_{\lambda}(\cdot)$ satisfies \eqref{mistura di esponenziali} for a survival function $S_{\lambda}(\cdot)$.

\begin{te} \label{Theorem: Governing equation equazione governante}
    Let  $M = \lg M(t),\ t\in\R_+ \rg$ be a Lévy process on $\mathbb{R}^n,\ n\in\N_+$. 
    Consider a process $Y = \{ Y(t),\ t\in\R_+ \}$, where $Y(t) \overset{d}{=} M\l \L _\lambda t \r$ for each $t\in\R_+$ and $u(\bm{x},t) := \E h(\bm{x} + Y(t))$ for $h\in \mathcal{S}\l \mathbb{R}^n \r$. 
    Then the map $t\longmapsto u(\bm{x},t)$  solves the Cauchy problem
    \begin{align}
       \mathcal{D}_t ^{\mu, -G} w(\bm{x},t)= -\lambda w(\bm{x},t) \qquad w(\bm{x},0) = h(\bm{x}),       \label{Governing equation equazione governante}
    \end{align}
    where the operator $\mathcal{D}_t ^{\mu, -G}$ is defined in Equation \eqref{Para Markov non-local operator}.
\end{te}

\begin{proof}
    Let us call $g(\cdot)$ the Lévy symbol of $M$; then the characteristic function of $M(t)$ reads
    \begin{align*}
        \E e^{i \langle \bm{\xi}, M(t)\rangle} = e^{t g(\bm{\xi})}, \qquad t\in\R_+, \quad \bm{\xi} \in \mathbb{R}^n.
    \end{align*}
    For $h\in \mathcal{S}\l\mathbb{R}^n \r,\ \bm{x}\in\R^n$, let $v(\bm{x},t) := \mathbb{E}h(\bm{x}+M(t))$ define the associated semigroup and  let $G$ be its infinitesimal generator.
    As $M$ is Lévy, one has (see Theorem 3.3.3 in \cite{applebaum2009levy} for details)
    \begin{align*}
        \hat{v}(\bm{\xi},t) =  \hat{h}\l\bm{\xi}\r e^{t g(\bm{\xi})}.
    \end{align*}
    By definition of $Y(t)$ and using a conditioning argument on the values of $\L_{\lambda}$, one obtains
    \begin{align}
        \hat{u}(\bm{\xi},t) =  \mathpzc{E} \hat{h}\l\bm{\xi}\r e^{t \L _\lambda g(\bm{\xi})} = \hat{h}\l\bm{\xi}\r S_{\lambda}\l -g\l\bm{\xi}\r t \r, \label{Teorema equazione: u tilde di xi e t}
    \end{align}
    where the last equality follows from $S_{\lambda}(t) = \mathpzc{E} e^{-\L_\lambda t}$ (see Proposition \ref{Proposition: uguaglianza in distribuzione del processo di conteggio}).
    \begin{comment}
    It follows from  
    \begin{align}
        \E e^{i\langle \bm{\xi}, M(\L t) \rangle} = \int_0^{\infty} e^{t g(\bm{\xi}) l} P(\L \in dl) = \E e^{tg(\bm{\xi})\, \L}= S(-g(\bm{\xi})t). \label{funzione caratteristica processo subordinato}
    \end{align}
    \end{comment}
   We now recall (see Equation \eqref{equazione autovalori}) that $S_{\lambda}(t)$ solves
    \begin{align*}
        \int_0^{\infty} \l S_{\lambda}(t) - S_{\lambda}(t-\tau) \r \mu(\tau) d\tau - \overline {\mu}(t) S_{\lambda}(0)= - \lambda S_{\lambda}(t), \qquad \forall t\in\R_+,
    \end{align*}
    where $S_{\lambda}(0)=1$.
    Given that it works for each $t$, it also holds for $at$, for each $a$, so that
    \begin{align*}
        \int_0^{\infty} \l S_{\lambda}(at) - S_{\lambda}(at-\tau) \r \mu(\tau) d\tau - \int _{at}^\infty \mu (\tau)d\tau = - \lambda S_{\lambda}(at) \qquad \forall t\in\R_+.
    \end{align*}
    By applying the change of variables $\tau  = aw$ in both integrals, we get
    \begin{align*}
        \int_0^{\infty} a\mu(aw) \, \l S_{\lambda}(at) - S_{\lambda}(a(t-w)) \r dw - a\int _t^\infty \mu (aw) dw = - \lambda S_{\lambda}(at) \qquad \forall t\in\R_+,
    \end{align*}
    namely $\mathcal{D}_t^{\mu,a} S_{\lambda}(at)=-\lambda S_{\lambda}(at)$, where the operator is defined in \eqref{non lo so}.
    By making the complex extension, we can apply this formula for $a = -g(\bm{\xi})$, to obtain
    \begin{align*}
        -\int_0^{\infty} g(\bm{\xi})\mu (-g(\bm{\xi})w)  \l  S_{\lambda}(-g(\bm{\xi})t)- S_{\lambda}(-g(\bm{\xi})(t-w)) \r \, dw & + g(\bm{\xi}) \int _t^\infty \mu (-g(\bm{\xi})w) dw \\ & = -\lambda  S_{\lambda}(-g(\bm{\xi})t).
    \end{align*}
    
    By multiplying both members by $\hat{h}\l\bm{\xi}\r$, applying the inverse Fourier transform, using Equation \eqref{Teorema equazione: u tilde di xi e t} and noting that $g(\bm{\xi})$ is the symbol of $G$, we have
    \begin{align*}
        -\int_0^{\infty} G \mu (-Gw)\, \l  u(x,t)- u(x, t-w) \r   dw + G \int _t^\infty \mu (-Gw) dw= -\lambda u(x,t).
    \end{align*}
    This concludes the proof.
\end{proof}

As a corollary, we present the result for non-Markovian chains with long-range dependence.

\begin{os}
    Consider a para-Markov chain $Y$ as in Definition \ref{Definizione definition di para markov}, with state space $\mathcal{E}$, and the notation $p_{ij}(t) := P(Y(t)=j|Y(0)=i),\ i,j\in\mathcal{E}$. As a consequence of Corollary \ref{Corollario: uguaglianza in distribuzione per le catene} and Theorem \ref{Theorem: Governing equation equazione governante}, the probability $p_{ij}(t)$ solves Equation \eqref{Governing equation equazione governante}. If $\mathcal{E}$ is finite, then $G$ is a matrix and then the operator \eqref{G mu(-Gtau) parte dell'operatore} reduces to $G\int_0^{\infty}e^{G\tau z} \kappa(dz)$. The latter case has been already studied in Paragraph 4 of \cite{facciaroni2025markov} when $S_\lambda (\cdot)$ is the Mittag-Leffler survival function. 
\end{os}

\begin{os}
    It is worth noting that in Equation \eqref{Governing equation equazione governante} the time and space operators cannot be separated, in general. It would be possible in the case of $H$ self-similar. It can be proved that the only self-similar subordinator is the $\alpha$-stable subordinator.  In this case the associated Bernstein function is $f(\bfn) = \bfn^{\alpha}$, the Lévy density is homogeneous of index $-(1+\alpha)$ and one gets
    \begin{align*}
        \l\frac{\partial}{\partial t} \r^{\al} u(x,t) = -(-G)^\al u(x,t), \qquad u(x,0)=h(x), \qquad  h\in \mathcal{S}\l\mathbb{R}^n \r.
    \end{align*}
    The latter is a special case of the abstract equation appeared in Theorem 4.1 in  \cite{facciaroni2025random}.
\end{os}

\subsection{Functional convergence to anomalous diffusion} \label{The case of the anomalous diffusion}

We now  introduce an anomalous diffusion process, which is related to para-Markov chains. This definition generalizes  the super-diffusion presented in \cite{facciaroni2025random}, the latter arising as hydrodynamic limit of a kinetic model.
In the second part of this subsection, we prove that this anomalous diffusion arises as a scaling limit of continuous-time random walks (CTRWs). The emergence of anomalous diffusion as a limit of CTRWs has been studied extensively within the framework of semi-Markov processes (see e.g. \cite{ScalasGorenfloMainardi2004,straka2011lagging, meerschaert2008triangular, Meerschaert2014}). In these works, the CTRW consists of i.i.d. spatial jumps separated by i.i.d. waiting times; in particular, the jump times are generated by a renewal process (see Remark \ref{Remark su Meerschaert} for details).
Here we go beyond this theory by assuming dependent waiting times between jumps, namely the counting process is of the type in Definition \ref{Definition of exchangeable fractional PP}.

For the sake of simplicity, we present the definition of the anomalous diffusion in the one dimensional case. We shall use the notation $a\wedge b$ for the minimum between $a,b\in\R$.

\begin{defin} \label{Definition: anomalous diffusion}
    Let us consider a stochastic process $Z = \lg Z(t),\ t\in\R_+ \rg$ on $\mathbb{R}$ and a survival function $S_{\lambda}:\R_+\longrightarrow[0,1]$ of the form \eqref{mistura di esponenziali}. $Z$ is said to be a para-Markov anomalous-diffusion process if, for any choice of times $0 \leq t_1 <\ldots < t_n$,  the vector $\l Z(t_1), \ldots, Z(t_n) \r$ has characteristic function
    \begin{align}
        \varphi_{n}(\bm{\xi}) = S_{\lambda} \l  \frac{1}{2} \langle \bm{\xi}, Q \bm{\xi} \rangle \r, \qquad \forall \bm{\xi}\in \R^{n},\ n\in\N_+, \label{funzione caratteristica finito dimensionali}
    \end{align}
    where the matrix $Q = \lq Q_{ij} \rq \in \R^{n\times n}$ has entries $Q_{ij} := t_i \wedge t_j,\ i,j \in \{1,\ldots, n\}$.
\end{defin}

As in the case of para-Markov chains (Definition \ref{Definizione definition di para markov}), parameters of the process will depend on the choice of $S_{\lambda}(\cdot)$. For $S_{\lambda}(t) = e^{-\lambda t}$, Definition \ref{Definition: anomalous diffusion} reduces to Brownian motion with var-cov matrix $\lambda Q$.
Otherwise, in many remarkable cases (see Theorem 3.2 in  \cite{orsingher2018semi}) the distribution $\nu _\lambda(\cdot)$ defined in \eqref{mistura di esponenziali} has infinite mean; hence, by using the relation between characteristic function and moments of the distribution, we can see that the variance of $Z(t)$ is infinite, so that $Z$ shows a superdiffusive behaviour. For example, in \cite{facciaroni2025random} the authors consider the case of $S_{1}(t) = E_{\alpha}\l- t^{\alpha}\r$ (see Example \ref{Esempio fractional PP, con distribuzione Mittag-Leffler});  the resulting process is superdiffusive and has a density solving the anomalous diffusion equation $\partial _t ^\alpha p(x,t)= -2^{-\alpha}(-\Delta)^\alpha p(x,t)$, with $\alpha \in (0,1)$, with the operator on the right side denoting the fractional Laplacian, under initial condition $p(x,0) = \delta(x)$.

\begin{prop} \label{uguaglianza fdd del mb lampertizzato}
    Let $B = \{B(t),\ t\in\R_+\}$ be a standard Brownian motion in $\mathbb{R}$ with var-cov matrix $Q = \lq Q_{ij} \rq$, where $Q_{ij} := t_i \wedge t_j,\ i,j \in \{1,\ldots, n\},\ n\in\N_+$.
    Then there exists a random variable $\L_{\lambda}$ independent of $B$ such that $ \{Z(t),\ t\in\R_+\} \overset{d}{=} \{B\l \L_{\lambda}t \r,\ t\in\R_+\}$.
\end{prop}

\begin{proof}
    From Formula \eqref{mistura di esponenziali}, the survival function $S_{\lambda}(\cdot)$ is completely monotone, then there exists a probability measure $\nu_{\lambda}(\cdot)$ such that
    \begin{align*}
        S_{\lambda}(t) = \int_0^{\infty}e^{-st} \nu_{\lambda}(ds), \qquad \forall t\in\R_+.
    \end{align*}
    Let us call $\L_{\lambda}$ a random variable whose distribution is $\nu_{\lambda}(\cdot)$.
    Let us consider a sequence of times $0 \leq t_1 <\ldots < t_n, \ n\in\N_+$. Consider the vector $\widetilde{\Gamma} := (B(\L_{\lambda} t_1),\ldots, B(\L_{\lambda} t_n))$, on some probability space, with random var-cov matrix given by $\L_{\lambda} Q$. Then, using $\E$ for the expectation on this probability space, one can write
    \begin{align*}
        \E \lq e^{i \langle \bm{\xi}, \widetilde{\Gamma}\rangle} \rq 
        % &= \E\lq e^{-\frac{1}{2} \langle \bm{\xi}, \opmoddue \bm{\xi} \rangle} \rq \\
        &= \mathpzc{E}\lq e^{-\frac{1}{2} \langle \bm{\xi}, \L_{\lambda} Q \bm{\xi} \rangle} \rq \\
        %&= \int_0^{\infty} e^{-\frac{1}{2} \langle \bm{\xi}, l Q \bm{\xi} \rangle} \, P\lq \L_{\lambda}\in dl \rq\\
        %&= \int_0^{\infty}e^{-\frac{1}{2} \langle \bm{\xi}, Q \bm{\xi} \rangle l} \, P\lq \L_{\lambda}\in dl \rq \\
        &= \mathpzc{E}\lq e^{-\frac{1}{2}\langle \bm{\xi}, Q \bm{\xi} \rangle \L_{\lambda}} \rq = S_{\lambda} \l \frac{1}{2}\langle \bm{\xi}, Q \bm{\xi} \rangle \r \qquad \forall \bm{\xi}\in \R^{n},
    \end{align*}
    where the matrix $Q \in\R^{n\times n}$ has the form of the one in Definition \ref{Definition: anomalous diffusion}.
\end{proof}

The above proposition shows that our anomalous diffusion has some relation with the models presented in \cite{di2019gaussian}.

\begin{cor}
     Consider the process $Z$ as in Definition \ref{Definition: anomalous diffusion} and assume $Q$ is positive definite. Then the vector $\l Z(t_1), \ldots, Z(t_n) \r$ has density given by
    \begin{align*}
        f_n(\bm{\xi}) &= \int_0^{\infty} \frac{1}{\l 2\pi\, s^{n} \det{Q}\r^{n/2}} e^{-\frac{1}{2s}\langle \bm{\xi},Q^{-1}\bm{\xi}\rangle}\nu_{\lambda}(ds), \qquad \bm{\xi}\in\R^n.
    \end{align*}
    Moreover, $Z(t)$ has a density solving
    \begin{align*}
        \mathcal{D}_t ^{\mu, -\Delta} p(x,t)= -\lambda p(x,t), \qquad p(x,0) = \delta(x).
    \end{align*}
\end{cor}

\begin{proof}
    The form of the density directly follows from Proposition \ref{uguaglianza fdd del mb lampertizzato}. The governing equation of the density of $Z(t)$ follows from Proposition \ref{uguaglianza fdd del mb lampertizzato} and Theorem \ref{Theorem: Governing equation equazione governante} setting $G = \Delta$.
\end{proof}

%\subsection{Convergence of the Continuous-Time Random Walk} \label{Convergence of the Compuond Poisson}

We now aim to find a CTRW approximation of the anomalous diffusion in Definition \ref{Definition: anomalous diffusion}.
Let $\{X_k,\ k\in\N_+\}$ be i.i.d.~random variables on $\R$ with $\E X_1=0, \ \E X_1^2  = 1$. The results can be easily generalized to the case of $\E X_1 \neq 0$ and $\E  \l X_1 - \E X_1 \r^2  = \si^2 < \infty$.
Let $\mathcal{N}$ be the counting process as in Definition \ref{Definition of exchangeable fractional PP} and assume $\mathcal{N}$ is independent of any $X_k,\ k\in\N_+$. 
Let us define the process $Y = \{Y(t),\ t\in\R_+\}$ where

\begin{align}
    Y(t)
    := \sum_{k=1}^{\mathcal{N}( t)} X_k, 
    \qquad t\in\R_+. \label{CTRW continuous-time random walk}
\end{align}
We shall use the notation $\implies$ to mean $J_1$-functional convergence in the Skorokhod topology in the space $D\l [0,\infty) \r$ of càdlàg functions on the real positive half line. For a light introduction to Skorokhod topologies we suggest to consult \cite{kern2022skorohod}.

\begin{te} \label{Theorem teorema : approssimazione diffusione anomala}
    Let $Y$ be defined as above and let us consider the rescaled process $Y_n := \{Y_n(t),\ t\in\R_+\}$, where
    \begin{align*}
        Y_{n}(t) := \frac{1}{\sqrt{n}}\sum_{k=1}^{\mathcal{N}(n t)} X_k, 
    \qquad t\in\R_+,
    \end{align*}
    for $n\in\N_+$. Then, in the limit $n\to +\infty$, one has
    \begin{align*}
    Y_{n}  \Longrightarrow  Z,
    \end{align*}
    where $Z$ is the anomalous diffusion of Definition \ref{Definition: anomalous diffusion}.
\end{te}

\begin{proof}
    To derive the thesis, it is sufficient to prove the convergence of finite-dimensional distributions and the tightness of the sequence of processes $\{Y_n\}$ (see Theorem 13.1 in \cite{billingsley2013convergence}). Before doing it, we observe that from Proposition \ref{Proposition: uguaglianza in distribuzione del processo di conteggio}, there exists a random variable $\L_{\lambda}$ such that $\mathcal{N}\l t\r \overset{d}{=}N(t\L_{\lambda}) $ for each $t\in\R_+$, being $N = \{N(t),\ t\in\R_+\}$ the Poisson process of intensity $1$. We call $\nu_{\lambda}(\cdot)$ the distribution of $\L_{\lambda}$.

    \begin{enumerate}
        \item We first prove that the finite dimensional distributions of $Y_n$ converge to those of $Z$. To this aim, we can study the convergence of the characteristic function of the random vector $\Gamma_n^k := \l Y_{n}(t_1), \ldots, Y_{n}(t_k) \r$ where $t_j \in \R_+,\ j = 1,\ldots, k$ and $k\in\N_+$. By using a conditioning argument and dominated convergence (since $\nu_{\lambda}(dl)$ integrates to 1), one has
        \begin{align*}
            \E e^{i \langle \bm{\xi}, \Gamma_n^k \rangle} &= \int_0^{\infty} \E \lq e^{i \langle \bm{\xi}, \Gamma_n^k \rangle} \mid \L_{\lambda}= l\rq \nu_{\lambda}(dl)\\
            &\overset{n\to\infty}{\longrightarrow} \int_0^{\infty} e^{-\frac{1}{2}  l \langle \bm{\xi}, Q \bm{\xi} \rangle} \nu_{\lambda}(dl) = S_{\lambda}  \l \frac{1}{2}   \langle \bm{\xi}, Q \bm{\xi} \rangle \r, \qquad \forall \bm{\xi} \in \R^n,
        \end{align*}
        for each $k\in\N_+,\ t_j \in \R_+,\ j = 1,\ldots, k$,
        where the matrix $Q = \lq Q_{ij} \rq \in \R^{n\times n}$ has entries $Q_{ij} := t_i \wedge t_j,\ i,j \in \{1,\ldots, n\}$.
        In the last step we used the fact that the characteristic function of the compound Poisson process converges to the one of Brownian motion due to Donsker's theorem.
        \item We now prove the tightness of the sequence $\{Y_n\}$. To this aim, one can use the sufficient condition proved by Aldous (see \cite{aldous1978stopping} or Theorem 16.10 in \cite{billingsley2013convergence}). Namely, we have to verify that both (16.22) and \textit{Condition 1°} in Paragraph 16 of \cite{billingsley2013convergence} hold.
        \begin{enumerate}
            \item To prove (16.22) it suffices to verify that for all $t\in\R_+$
            \begin{align}
                \lim_{\gamma \to \infty}\sup_{n\in\N_+} P(|Y_n(t)|> \gamma) = 0, \label{Proof teo convergenza parte 2 punto a) - eq da provare}
            \end{align}
            which prevents the escape of mass to infinity.
           For any $K > 0$ one has
            \begin{align}
                P(|Y_n(t)|> \gamma)
                %&= \E 1_{\{|Y_n(t)|> \gamma\}} 
                %= \E \lq (1_{\{K \leq \L\}} + 1_{\{K > \L\}})  \E \lq 1_{\{|Y_n(t)|> \gamma\}} \mid \L_{\lambda}\rq \rq \\
                &= \int_0^K  P(|Y_n(t)| > \gamma \mid \L_{\lambda}= l)\nu_{\lambda}(dl) \nonumber \\
                &+ \int_K^{\infty} P(|Y_n(t)| > \gamma \mid \L_{\lambda}= l) \nu_{\lambda}(dl). 
                %&=  \E \lq 1_{\{K \leq \L\}}  P(|Y_n(t)|> \gamma \mid \L) \rq +  \E \lq 1_{\{K > \L\}}  P(|Y_n(t)|> \gamma \mid \L) \rq 
                \label{Proof teo convergenza parte 2 punto a) - funzionale}
            \end{align}
            Given that $P(|Y_n(t)|> \gamma \mid \L_{\lambda}= l) \leq 1$ the following inequality holds
            \begin{align*}
                P(|Y_n(t)|> \gamma) \leq    \int_0^K P\l |Y_n(t)|> \gamma \mid \L_{\lambda}= l\r\nu_{\lambda}(dl)   + P(\L_{\lambda}> K).
            \end{align*}
            By simple calculations\footnote{Let us consider a sequence $\{X_k,\ k\in\N_+\}$ of i.i.d. random variables with zero mean and $\mathrm{Var}(X_1) < \infty$. Let us introduce the random variable $N \in \N$ with $\E N < \infty$ and $N$ independent of $X_k$, for each $k\in\N_+$. Then $\mathrm{Var} \lq  \sum_{k = 1}^N X_k \rq = \mathrm{Var}[X_1] \E N$.}, we have 
            \begin{align*}
                \mathrm{Var}(Y_n(t) \mid \L_{\lambda}= l) = \frac{ 1 }{n} \E \lq \mathcal{N}(nt) \mid \L_{\lambda}= l \rq =   l t, 
            \end{align*}
            so that we can use the conditional Chebyshev inequality to get 
            \begin{align*}
                P(|Y_n(t)|> \gamma \mid \L_{\lambda} = l) \leq  \frac{  l  t}{\gamma^2}. 
            \end{align*}
            In conclusion, from Equation \eqref{Proof teo convergenza parte 2 punto a) - funzionale}, 
            \begin{align*}
                P(|Y_n(t)|> \gamma) \leq \frac{t}{\gamma^2} \int_0^K l \nu_{\lambda}(dl) + P(\L_{\lambda}> K).
            \end{align*}
            The right hand side does not depend on $n$, therefore 
            \begin{align*}
                \sup_{n\in\N_+} P(|Y_n(t)|> \gamma) \leq \frac{t}{\gamma^2} \int_0^K l \nu_{\lambda}(dl) + P(\L_{\lambda}> K) .
            \end{align*}
            The inequality holds for each $K > 0$. Passing to the limit $\gamma\to \infty$ and then $K \to \infty$, one gets Equation \eqref{Proof teo convergenza parte 2 punto a) - eq da provare} as required.
            \begin{comment}
            Then, for each $\ve > 0$ we can choose $\gamma$ and $K$ such that  $\ve - P\l\L_{\lambda} > K\r > 0$ and
            \begin{align*}
                \ve = \frac{t}{\gamma^2} \int_0^K l \nu_{\lambda}(dl) + P(\L_{\lambda}> K),
            \end{align*}
            to have 
            \begin{align*}
                \sup_{n\in\N_+} P(|Y_n(t)|> \gamma) \leq \ve.
            \end{align*}
            \end{comment}
            \item Let $\tau  = \{\tau_n,\ n\in\N_+\}$ be a sequence of stopping times almost surely finite. Let us consider the increments
            \begin{align*}
                \Delta_n := Y_n(\tau_n + \zeta_n) - Y_n(\tau_n) = \frac{1}{\sqrt{n}} \sum_{k = \mathcal{N}(n\tau_n) + 1}^{\mathcal{N}(n(\tau_n+ \zeta_n))} X_k,
            \end{align*}
            where $\zeta_n \to 0$ for $n\to\infty$.
            \textit{Condition 1°} in Paragraph 16 in \cite{billingsley2013convergence}, which corresponds to Condition A in \cite{aldous1978stopping}, requires to prove
            \begin{align*}
                \lim_{n \to \infty} P(|\Delta_n | > \gamma) = 0, \qquad \text{for each} \ \gamma > 0,
            \end{align*}
            which ensures that, at any random time, arbitrarily small time increments cannot produce large jumps of the limiting process. The jumps $X_k$s are i.i.d. so that $\Delta_n$ has the same distribution as $\sum_{k = 1}^{\mathcal{N}(n\l\tau_n + \zeta_n\r) - \mathcal{N}(n\tau_n)}X_k$.
            By stationarity of increments, we observe that, conditionally to  $\L_{\lambda}= l$, the random variable $\mathcal{N}(n(\tau_n+\zeta_n))-\mathcal{N}(n\tau_n)$ follows a Poisson distribution of intensity $n l\zeta_n$.
            Hence 
            \begin{align*}
                \mathrm{Var}(\Delta_n\mid\L_{\lambda}=l)
            =  l \zeta_n.
            \end{align*}
             By using the conditional Chebyshev inequality, one has
            \begin{align*}
            P\l |\Delta_n|>\gamma\mid\L_{\lambda}=l \r
            \le \frac{ l\zeta_n}{\gamma^2}.
            \end{align*}
            Proceeding similarly to point (a), for each $K > 0$ the following holds
            \begin{align*}
                P(|\Delta_n|>\gamma)
                \leq \frac{\zeta_n}{\gamma^2} \int_0^K l \nu_{\lambda}(dl)
                + P(\L_{\lambda}>K).
            \end{align*}
            Passing to the limit $n\to\infty$ and then $K \to \infty$, one gets the thesis. 
        \end{enumerate}
    \end{enumerate}
\end{proof}

\begin{os} \label{Remark su Meerschaert}
    As mentioned at the beginning of this subsection, convergence of CTRWs to anomalous diffusion has been widely studied in the semi-Markov setting. In the latter case, the position of the problem is the same as in the present paper, except for the fact that the counting process $\mathcal{N}$ in \eqref{CTRW continuous-time random walk} is of renewal type. The limiting anomalous diffusion is the time-changed process $B(\Inv(t))$, where $\Inv$ is the inverse of a subordinator; the density $p(x,t)$ of this process is governed by the non-local equation 
    \begin{align*}
        \mathcal{D}_t^{\mu}p(x,t) = \frac{1}{2}\Delta p(x,t), \qquad p(x,0) = \delta(x).
    \end{align*}
    In the special case where the waiting times of $\mathcal{N}$ are in the basin of attraction of a $\alpha$-stable law, with $\alpha \in (0,1)$, then $\Inv(t)$ is an inverse $\alpha$-stable subordinator, with $\alpha \in (0,1)$. Here  $\{B(\Inv(t)),\ t\in\R_+\}$ is a sub-diffusion, with mean square displacement $Ct^\alpha$; the process is governed by the fractional heat equation
    \begin{align*}
        \l \frac{\partial }{\partial t} \r^{\alpha} p(x,t) = \frac{1}{2}\Delta p(x,t), \qquad p(x,0) = \delta(x).
    \end{align*}
\end{os}

\section{Beyond semi-Markov and para-Markov chains} \label{Section: schema concentrico}

% \textcolor{red}{Consider $\tilde{M}, M, \si, \Inv, Y$ come in Section \ref{}}. Oppure

As recalled in the previous section, an important class of semi-Markov chains is given by the time change of Markovian ones with the inverse of Lévy subordinators. Moreover,
in Corollary \ref{Corollario: uguaglianza in distribuzione per le catene} we have seen that, in the sense of finite-dimensional distributions, para-Markov chains are equal to a random linear time change of Markov ones.
Here we propose to frame these relevant facts in a more general framework.

We thus consider a broad class of non-Markovian chains whose long-memory behaviour arises from dependence among waiting times. These processes are defined by time-changing a continuous-time Markov chain with the inverse of a suitable increasing process $\sigma$. The knowledge of the finite-dimensional distributions of $\sigma$ determines the joint distribution of the waiting times. We first show that, besides semi-Markov and para-Markov chains, also several well-known models from the literature (such as those with Schur-constant or exchangeable waiting times) fall within this theory (see Subsection \ref{Models  with dependent waiting times}). Then, in Subsection \ref{Section o subsection Stable stabili}, we introduce  new models obtained by assuming that $\sigma$ is an increasing stable process. 
%Subsection \ref{Self-similar stable processes} is devoted to particular $\si$ belonging to subclasses of stable processes.

\subsection{Chains with dependent waiting times} \label{Models  with dependent waiting times}

Consider a continuous-time Markov chain $M$ as defined by  (\ref{continuous-time chain M(t)}), with embedded chain $\tilde{M}$ and i.i.d. exponential waiting times $\{W_k,\ k\in\N_+\}$.
Consider a right-continuous, increasing process $\sigma = \{\sigma(t),\, t \in\R_+\}$, independent of $M$, with no deterministic points of discontinuity and with $\sigma (0)=0$ almost surely. Let $\Inv = \{\Inv(t),\, t \in\R_+\}$ denote the inverse of $\si$, defined by
    \begin{align*}
        \Inv(t) := \inf\lg s\in\R_+ \mid \sigma(s) > t\rg, \qquad t\in\R_+.
    \end{align*}
    
We want to analyse the properties of the process $Y = \{Y(t),\ t\in\R_+\}$, where $Y(t) := M(\Inv(t))$ for each $t\in\R_+$.
Let us call $\bm{J}_n:=(J_1,\ldots, J_n),\ n\in\N_+$ the first $n$ waiting times of $Y$.
\begin{comment}
\begin{prop} \label{time changee}
    Let $M$, $\sigma$, $\psi$ and $Y$ be defined as above. Assume that  $\sigma (0)=0$ almost surely. Then $Y$ has the same embedded chain $\tilde{M}$ of $M$, namely
    
        \begin{align*}
        Y(t) = \tilde{M}(n), \qquad \sum _{i=1}^n J_i \leq t<
\sum _{i=1}^{n+1}J_i    \end{align*}
   
    with waiting times satisfying
    \begin{align*}
        J_n = \si \l T_{n+1} \r - \si \l T_n^- \r.
    \end{align*}
\end{prop}

\begin{proof}
    Since $\sigma$ is a.s. increasing, then $\psi$ has continuous sample paths. This implies that $M(\psi (t))$ has the same embedded chain of $M$. Far vedere il time change ecc ecc.
\end{proof}

A special case is given by the following immediate corollary
\begin{cor} \label{coroll}
If $\sigma$ has stationary increments, then the waiting times $J_n$ all have the same distribution, namely $J_1 \overset{d}{=}\sigma (W_1)$
\end{cor}
\begin{proof}
    Mettere un paio di righe di proof.
\end{proof}

\end{comment}
The time-changed process $Y$ has the same embedded chain $\tilde{M}$ of $M$, namely
    \begin{align*}
    Y(t) = \tilde{M}(n), \qquad \sum _{k=1}^n J_k \leq t<
    \sum _{k=1}^{n+1}J_k  , \qquad n\in\N_+. \end{align*}
    This follows directly from the fact that $\sigma$ is strictly increasing and  $\psi$ has continuous sample paths.
    
    We claim that, if the finite-dimensional distributions of $\sigma$ are known, then the process $Y$ is analytically tractable and it is possible to find the joint distribution of waiting times.
    
    For $n\in\N_+$, we shall use the notation $ A_n := I_n - K_n $ where $I_n$ is the identity matrix of dimension $n$, $K_n$ is the lower shift matrix of dimension $n$, i.e.~ $(K_n)_{ij}$ equals 1 for $i = j+1$ and 0 otherwise. Moreover, $B_n$ shall indicate an $n$-dimensional lower triangular matrix with entries on and below the main diagonal equal to 1. Finally, $ \bm{\si}(\bm{t}_n) :=  \l \sigma(t_1),\ldots, \sigma(t_n) \r^T$, where $\bm{t}_n := \l t_1, \dots, t_n \r$.

\begin{te} \label{distribuzione intertempi}
    Let $Y$ be defined as above.
    % Let $M = \{M(t),\ t\in\R_+\}$ be a continuous-time Markov chain, having i.i.d. exponential waiting times $\{W_k,\ k\in\N_+\}$ of mean $\lambda ^{-1}$. Assume $\si$ as before. Let us consider $Y = \{Y(t),\ t\in\R_+\}$ where $Y(t) := M(\Inv(t))$ for each $t\in\R_+$ and $\Inv = \{\Inv(t),\ t\in\R_+\}$ is the inverse process of $\si$. Let us denote by $\{J_k,\ k\in\N_+\}$ the waiting times of $Y$.
    Then, the waiting times are given by
    \begin{align}
            \lg J_n,\ n\in\N_+\rg = \lg \si \l T_{n+1} \r - \si \l T_n^- \r, \ n\in\N_+ \rg, \label{forma del n waiting time di Y(t) = M(psi(t))}
        \end{align}
        for $T_n := \sum_{k = 1}^n W_k$.
    Moreover, for each $n\in\N_+$, the random vector $\bm{J}_n$ has distribution
        \begin{align*}
        \E e^{i \langle \bm{\xi}, \bm{J}_n \rangle} = \lambda^n\int_{\R^n_+} \vp_{\bm{\si}(B_n \bm{w}_n)} \l A_n^T \bm{\xi} \r  e^{-\lambda \sum_{j = 1}^nw_j} dw_1 \ldots dw_n, \qquad \bm{\xi} \in\R^n,
    \end{align*}
    where $\bm{w}_n := \l w_1,\ldots, w_n\r^T$ and $\vp_{\bm{\si}(\bm{t}_n)}(\cdot)$ is the characteristic function of the vector $\bm{\si}(\bm{t}_n)$.
\end{te}

\begin{proof}
    The first part follows from the definition of $Y$. Indeed, $M(\Inv(t)) = \tilde{M}(n)$ for $ \Inv(t) \in \lq T_n, T_{n+1} \r$ so that Equation \eqref{forma del n waiting time di Y(t) = M(psi(t))} can be obtained by applying $\si$ to each member and using the strictly increasing property of $\si$.
    
    Let $\bm{t}_n := (t_1, t_2, \dots, t_n)^T$ and consider the vectors
    \begin{align*}
       \bm{\si} (B_n\,\bm{t}_n) &= \l \si(t_1), \si(t_1+t_2), \ldots, \si \l\sum_{j = 1}^n t_j\r \r
       \end{align*}
       and
       \begin{align*}
       A_n \,\bm{\si} (B_n \, \bm{t}_n) &=  \l \si(t_1), \si(t_1+t_2) - \si(t_1), \ldots, \si \l\sum_{j = 1}^n t_j\r - \si \l \sum_{j = 1}^{n-1} t_j \r \r.
    \end{align*}
    Now, let $\bm{W}_n :=(W_1, W_2, \dots, W_n)^T$ be the vector of the first $n\in\N_+$ i.i.d. waiting times of the Markov process $M$.
By using Formula (\ref{forma del n waiting time di Y(t) = M(psi(t))}), we see that 

\begin{align*}
    \bm{J}_n=  A_n \bm{\si} (B_n \, \bm{W}_n) . 
\end{align*}
Using the properties of the scalar product, we obtain
\begin{align*}
     \langle\bm{\xi},  \bm{J}_n   \rangle  = \langle\bm{\xi},       A_n \bm{\si} (B_n \, \bm{W}_n)              \rangle=
     \langle A_n^T\bm{\xi}, \bm{\si}( B_n \, \bm{W}_n)  \rangle , \qquad \bm{\xi} \in \R^n.
\end{align*}

Then, using a standard conditioning argument one obtains
\begin{align*}
        \E e^{i \langle \bm{\xi}, \bm{J}_n \rangle} = \lambda^n\int_{\R^n_+} \vp_{\bm{\si}(B_n  \bm{w}_n)} \l A_n^T \bm{\xi} \r  e^{-\lambda \sum_{j = 1}^nw_j} dw_1 \ldots dw_n, \qquad \bm{\xi} \in\R^n.
    \end{align*}
    %where $\vp_{\bm{\si}(\bm{t}_n)}(\cdot)$ is the characteristic function of the vector $\bm{\si}(\bm{t}_n)$.
\end{proof}

\begin{comment}
\begin{os}
    It is worth noting that even if $\si$ is non stable we still have the following distribution for the vector of waiting times of $Y$
    \begin{align*}
        \E e^{i \langle \bm{\xi}, \bm{J}_n \rangle} = \int_{\R^n_+} \vp_{B_n \, \bm{\si}(\bm{w}_n)} \l A_n^T \bm{\xi} \r \lambda^n e^{-\lambda \sum_{j = 1}^nw_j} dw_1 \ldots dw_n, \qquad \bm{\xi} \in\R^n,
    \end{align*}
    where we have used the notation of the proof, and $\vp_{\bm{\si}(\bm{t}_n)}(\cdot)$ is the characteristic function of the vector $\bm{\si}(\bm{t}_n)$ as in Formula \eqref{Funzione caratteristica processo stabile}.
\end{os}
\end{comment}

\subsubsection{The case of $\si$ with stationary increments} \label{The case of sigma with stationary increments}

As a relevant consequence of Equation \eqref{forma del n waiting time di Y(t) = M(psi(t))}, if $\si$ has stationary increments,, i.e.~$\sigma (t+h)-\sigma (h) \overset{d}{=}\sigma (t)$ for all $(t, h)\in\R_+^2$, then all the waiting times $J_n$ have the same distribution, namely $J_1 \overset{d}{=}\sigma (W_1)$.

%A first example is given by para-Markov chains, studied in Section \ref{Section: para makov chains}; indeed, in the sense of Proposition \ref{Proposition: uguaglianza in distribuzione del processo di conteggio}, they are constructed by assuming a linear in time process $\sigma (t)=t/\mathcal{L}_\lambda$ (which has stationary but dependent increments) for a suitable positive random variable $\mathcal{L}_\lambda$, whose distribution is related to complete Bernstein functions.

Many interesting models with dependent waiting times fall in this theory. A remarkable example is given by  non-Markovian chains whose waiting times have any Schur constant distribution.
For convenience, we recall the notion of Schur constant sequences of random variables.
A sequence of non-negative random variables $\lg Z_k,\ k\in\N_+\rg$ is said to be  of Schur-constant type if there exists a survival function $S:\R_+ \longrightarrow [0,1]$ such that
    \begin{align}
        P \l Z_1 > t_1, \ldots, Z_n > t_n \r = S\l \sum_{k = 1}^n t_k\r \qquad \forall n\in\mathbb{N},\, \forall t_k \in\R_+,\ k \in \{ 1,\ldots, n\}. \label{Equation Schur constant sequence}
\end{align}
    
The joint survival probability only depends  on the sum of the waiting times a.k.a.~lifetimes, which makes the model statistically tractable and suitable for many applications (see e.g.~\cite{barlow1992finetti, caramellino1994dependence, caramellino1996wbf}). In order to embed these models in the above time-change setting, we must distinguish between two cases, depending on whether $S(\cdot)$ is completely monotone or not.

    \begin{itemize}
        \item If $S(\cdot) $ is completely monotone, then it is true that $S (t) = \int_{0}^{\infty} e^{-st} \mu (ds)$ 
        for a unique probability measure $\mu$. Hence $S(\cdot)$ defines a mixture of exponential distributions.  By re-adapting Proposition (\ref{Proposition: uguaglianza in distribuzione del processo di conteggio}), we see that this case is consistent with $\sigma (t)= t/\mathcal{L}$, which has stationary but dependent increments, and then   $\Inv(t) = t\L$,  where $\L$ has distribution $\mu$. By Corollary \ref{Corollario: uguaglianza in distribuzione per le catene}, para-Markov chains are included, in the sense of finite-dimensional distribution, in this class of time-changed Markov processes.
        %This implies that, in the sense of finite-dimensional distributions, our non-Markovian chain is equivalent to   $M(\mathcal{L}t)$.
        \begin{comment}Please note that in the case of para-Markov chains there were restrictions on $\L$. As an example, the case of $U \sim \mathrm{Unif}(0,1)$ is included in the framework of Schur constant waiting times with $S(\cdot)$ representable as a mixture of exponentials, but it cannot exist a para-Markov chain that is equal in distribution to $M(tU)$.
        \end{comment}
        \item If $S(\cdot)$ is not completely monotone, then it cannot be written as a mixture of exponential distributions (consider, e.g.,  $S(\cdot)$ to be the survival function of a Gamma random variable). In this case, the structure of $\sigma$ is more complicated than before. A sufficient, but not necessary, condition for the waiting times to be Schur constant is that $\sigma$ has Schur constant increments, i.e.
        \begin{align*}
        &P(\sigma (t_1)> x_1, \sigma (t_2)-\sigma (t_1) >x_2, \dots, \sigma (t_n)-\sigma (t_{n-1}) >x_n)
        \\ &= g\biggl (\sum _{i=1}^n x_i, t_1, t_2, \dots, t_n \biggr )
        \end{align*}
        for a suitable function $g$. This can be verified using Equation \eqref{forma del n waiting time di Y(t) = M(psi(t))} and a standard conditioning argument.
    \end{itemize}

In the same way as the waiting times of para-Markov chains are a particular case of Schur constant sequences, we observe that 
Schur-constant sequences are a particular case of exchangeable ones. See \cite{de1929funzione} (in Italian) for the very first rigorous analysis of such sequences and, for instance, \cite{diaconis1977finite,diaconis1980finite,kingman1978uses} and references therein for further developments. Recently, exchangeability has been exploited to characterize measure-valued Pólya urn sequences \cite{sariev2024characterization}. We recall that the random variables \( (Z_1, \dots, Z_n) \)  are said to be exchangeable if \begin{align*}
        P \l Z_1 > t_1, \ldots, Z_n > t_n \r = P \l Z_{\pi(1)} > t_1, \ldots, Z_{\pi(n)} > t_n \r
    \end{align*}
    for each $\pi$ in the symmetric group of order $n!$, i.e.~$(Z_1,\ldots, Z_n) \overset{d}{=} \l Z_{\pi(1)},\ldots, Z_{\pi(n)} \r$. For example, consider a joint survival function depending on the sum, or the product, or the maximum of the $t_i$s (using the sum we get the Schur-constant case).
\begin{comment}
    If the joint survival function of $(Z_1,\ldots, Z_n)$ is absolutely continuous, one has the following joint density of the first $n$ waiting times
    \begin{align*}
        g(t_1, t_2,\ldots, t_n) = (-1)^n \frac{\partial ^n}{\partial t_1 \cdots \partial t_n} S\l \sum_{j = 1}^n t_j\r, \qquad n\in\mathbb{N}.
    \end{align*}
\end{comment}
It can be seen that a sequence of random variables is exchangeable if and only if its characteristic function is symmetric in its arguments, i.e.~$\E e^{i \langle \xi, (Z_1,\ldots, Z_n) \rangle} = \E e^{i \langle \pi(\xi), (Z_1,\ldots, Z_n) \rangle}$
for each $\pi$.
    The exchangeability of the waiting times $(J_1,\ldots, J_n)$ ensures that all the waiting times have the same distribution (the opposite is not true). Hence, models with exchangeable waiting times fall in the framework of $\sigma$ having stationary increments.

    For the sake of clarity, we sum up the ideas in the following scheme in Figure \ref{mo vedemo}.

\begin{figure}[h!]
    \begin{center}
    \includegraphics[height = 6cm, width = 11cm]{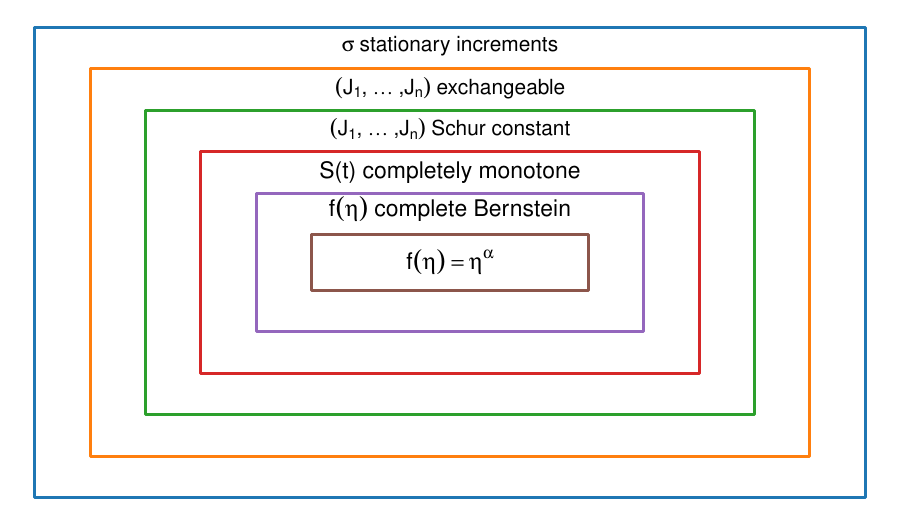}
    \caption{Comparison between different time changes: $\si$ is an increasing non-negative process and $(J_1, \ldots, J_n)$ is the vector of the first $n$ waiting times of the time changed process. The survival function of each $J_k$ is $S(\cdot)$ and $f(\cdot)$ is the Bernstein function associated to a subordinator $H$ with inverse $L$, such that $S(t) = \E e^{-\lambda t}$.}
    \label{mo vedemo}
    \end{center}
\end{figure}

\subsection{Time change by inverses of increasing stable processes} \label{Section o subsection Stable stabili}
%As said in Subsection \ref{Models  with dependent waiting times}, if the finite-dimensional distributions of $\sigma$ are known, it is possible to start from a Markovian process $M = \{M(t),\ t\in\R_+\}$ and define the subordinated process $\lg M(\Inv(t)),\ t\in\R_+ \rg$, where $\Inv = \{\Inv(t),\ t\in\R_+\}$ is the inverse process of $\sigma$.
One of the most remarkable cases of processes with  manageable finite-dimensional distributions is given by stable processes. Therefore, we can construct new chains with dependent waiting times starting from them.
We note that the case where $\sigma$ is a stable subordinator has been widely considered in the literature on semi-Markov processes. However, the case where $\sigma$ is a general increasing stable process has never been considered to the best of our knowledge. In the latter case we can also go beyond the semi-Markov setting. 
%Indeed, here we do not impose, at least at the beginning, that $\si$ has independent or stationary increments. 

A standard reference for stable processes is \cite{samorodnitsky1994stable}. Here we recall some basic facts.
A random vector $\bm{X} = \l X_1,\ldots, X_n \r$ is said to be stable if for any $a,b\in\R_+$ there exists $c\in\R_+$ and $\bm{e} \in \R^n$ such that
$ a \bm{X}^{(1)} + b\bm{X}^{(2)} \overset{d}{=}c\bm{X} + \bm{e} $
where $\bm{X}^{(i)},\ i = 1,2$ are independent copies of $\bm{X}$. It can be proved that $c = \l a^\al + b^\al \r^{\frac{1}{\al}}$ for some $\al\in(0,2]$, so that we refer to $\bm{X}$ as an $\al$-stable vector.

The distribution of a $n$-dimensional stable vector is defined by the following: the stability parameter $\al\in(0,2]$ which controls tail heaviness and the existence of moments of the marginal components (for $\al = 2$ the Gaussian case is recovered); a measure $\Gamma$ on the unit sphere $\mathcal{S}^{n-1}$ which is called the spectral measure and controls the dependence structure between components (for $n = 1$ it is related to the skewness parameter); a vector $\mu \in \R^n$ which represents the location parameter (for $\al > 1$, it represents the expectation of $\bm{X}$).

We shall consider the case of $\mu = 0\in\R^n$ to avoid drifts (for the sake of simplicity). Furthermore, we shall restrict to $\al \in (0,1)$ as explained below. In this case, the characteristic function of an $\al$-stable vector of dimension $n\in\N_+$ reads (see Theorem 6.17 in \cite{meerschaert2019stochastic} or re-arrange Theorem 2.3.1 in \cite{samorodnitsky1994stable}) 

\begin{align}
    \vp_{\bm{X}}(\bm{\xi}) :=   \exp\lg- \int_{\mathcal{S}^{n-1}} (-i\langle \bm{\xi}, \bm{s} \rangle )^{\al}  \Gamma (d\bm{s})\rg, \qquad \bm{\xi} \in \R^n. \label{Funzione caratteristica vettore stabile}
\end{align}

A real valued process $\sigma = \{\sigma(t),\ t\in\R_+\}$ is said to be $\alpha$-stable if all its finite-dimensional distributions are $\alpha$-stable; namely,  for each choice of $t_k\in\R_+,\ k \in\{ 1,\ldots, n\}$ and $n\in\N_+$, such that $t_1<t_2<...<t_n$, the vector  $ \bm{\si}(\bm{t}_n) :=  \l \sigma(t_1),\ldots, \sigma(t_n) \r^T$ is stable in $\R^n$, where $\bm{t}_n := \l t_1, \dots, t_n \r$. Here the spectral measure $\Gamma(d\bm{s}) = \Gamma _{\bm{t}_n} (d\bm{s})$  includes the dependence on the vector of times.

We are interested in almost surely increasing processes. Hence $\bm{\si}(\bm{t}_n)$ is supported on $C :=\{(x_1, \dots, x_n) \in \mathbb{R}_+^n \mid  x_k < x_{k+1},\ k  \in \{1,\ldots, n-1\} \}$.
This means that the spectral measure is compactly supported on  
$\Omega := \{ \bm{s} \in S^
{d-1} \mid  r\bm{s} \in C,\  \textrm{for} \,\, r \in \R_+ \} $. 
In this case the characteristic function \eqref{Funzione caratteristica vettore stabile} becomes 

\begin{align}
    \vp_{\bm{\si}(\bm{t}_n)}(\bm{\xi}) :=   \exp\lg- \int_{\Omega} (-i\langle \bm{\xi}, \bm{s} \rangle )^{\al}  \Gamma_{\bm{t}_n} (d\bm{s})\rg, \qquad \bm{\xi} \in \R^n. \label{Funzione caratteristica processo stabile}
\end{align}
As before, we also assume that $\sigma (0)=0$ almost surely.
We then emphasize that this choice of $\Gamma_{\bm{t}_n}(\cdot)$ ensures that $\al \in (0,1)$, so that the process $\si$ is positive almost surely.

As in the previous sections, we require $\si$ not to have deterministic times of discontinuity. 
Under these constraints, the inverse process $\Inv$ is well defined.
So, letting $M$ be a Markov process independent of $\sigma$, we shall study the time-changed process $Y = \{Y(t),\ t\in\R_+\}$ where
\begin{align*}
    Y(t) := M(\Inv(t)), \qquad t\in\R_+.
\end{align*}
The case where $M$ is a Brownian motion leads to an anomalous diffusion with long range-dependence, which goes beyond the semi-Markov anomalous diffusion described by equation (\ref{kkk}), the latter based on the inverse of a Lévy stable subordinator.
On this point, we conjecture that such anomalous diffusion should arise as scaling limit of CRTWs with dependent waiting times, lying in the basin of attraction of a stable law, but this analysis should be rather complicated and goes beyond the scope of the present paper. However, 
we here consider  the case where $M$ is a continuous-time Markov chain as in Section \ref{Models  with dependent waiting times}, with i.i.d. exponential waiting times $\{W_k,\ k\geq1\}$ of expectation $\lambda ^{-1}$. Hence the subordinated process $Y$ has dependent waiting times $\{J_k,\ k\geq1\}$ defined by Formula \eqref{forma del n waiting time di Y(t) = M(psi(t))}. 

The following result is a corollary of Theorem \ref{distribuzione intertempi}.

\begin{cor}
    Under the hypothesis of Theorem \ref{distribuzione intertempi}, if $\si$ is $\al$-stable, then the random vector $\bm{J}_n$ has distribution
    \begin{align}
      \E e^{i \langle\bm{\xi},  \bm{J}_n   \rangle} =
      \lambda^n\int_{\R^n_+} \exp \lg -\int_{\Omega} \l -i \langle A_n^T \bm{\xi}, \bm{s} \rangle \r^{\al} \,\Gamma_{B_n \, \bm{w}_n}(d\bm{s})-\lambda \sum_{j=1}^n w_j \rg dw_1\ldots dw_n. \label{funzione caratteristica intertempi nel caso stabile}
    \end{align}
\end{cor}
\begin{proof}
    The proof follows from Theorem \ref{distribuzione intertempi} and the characteristic function of stable processes \eqref{Funzione caratteristica processo stabile}.
\end{proof}

\begin{os}
Consider a stable process $\sigma$ having stationary increments. This condition does not uniquely identify the finite-dimensional distributions \eqref{Funzione caratteristica processo stabile}. Then, the joint distribution of the waiting times of $Y$ is not uniquely determined. However, as seen at the beginning of \ref{The case of sigma with stationary increments}, all the waiting times of $Y$ have the same distribution as $\sigma (W_1)$. As a remarkable example, assume that 
\begin{align*}
    \mathbb{E}e^{i\xi\sigma (t)}=e^{-t(-i\xi)^\alpha}, \qquad \xi \in \R,\ t\in\R_+
\end{align*}
i.e.~the characteristic exponent is linear in the time variable. Here, the waiting times of $Y$ all have the same Mittag-Leffler distribution but are dependent. Indeed, by a conditioning argument one derives
\begin{align*}
    \E e^{i \xi \si(W_1)} = \lambda \int_0^{\infty} e^{-w(-i\xi)^{\alpha}}e^{-\lambda w} dw = \frac{\lambda}{\lambda + \l -i\xi\r^{\alpha}},
\end{align*}
which is the characteristic function of the Mittag-Leffler distribution of Example \ref{Esempio fractional PP, con distribuzione Mittag-Leffler}.
This is analogous to what happens to para-Markov chains \cite{facciaroni2025markov}, but here we do not  have a Schur constant dependence.
If $\sigma$ also has independent increments, then $\sigma$ is a stable subordinator and the waiting times of $Y$ are i.i.d.~Mittag-Leffler. 
\end{os}

\begin{comment}
\begin{os} \label{Remark on stable non negative non decreasing process}
    We observe that such a process exists. Indeed, starting from the characteristic function of a generic $\si$ $\al$-stable process \eqref{Funzione caratteristica processo stabile}, it is sufficient to impose $\al \in (0,1)$ to avoid oscillatory sample paths and to require the spectral measure to take values on the positive half-line. In the one dimensional case, it would correspond to the request $\beta = 1$ that is the distribution totally skewed to the right.
\end{os}
\end{comment}

\begin{comment}
    As an example, consider $M_{\al}$ an one-side $\al$-stable Lévy measure with $\al < 1$ and consider kernels $f_t(x) \geq 0$ such that
    \begin{align*}
        t \leq s \implies f_t(x) \leq f_s(x) \quad \text{almost everywhere.}
    \end{align*}
    Then the sought process is given by $X = \{X(t),\ t\in\R_+\}$ where
    \begin{align*}
        X(t) = \int_0^{\infty} f_t(x) dM_{\al}(x), \qquad \forall t\in\R_+.
    \end{align*}
    As a special case, if $f_t(\cdot)$ is deterministic we obtain a Lévy process.
\end{comment}

\subsubsection{Self-similar stable processes} \label{Self-similar stable processes}

Following \cite{samorodnitsky1994stable}, a central role is played by self-similar stable processes. We recall that a process $\si = \{ \si(t),\ t\in\R_+\}$ is said to be self-similar with Hurst index $\mathcal{H}$ if for each $c\in\R_+$ one has
\begin{align*}
    \l \si(ct_1),\ldots, \si(ct_n) \r \overset{d}{=} c^{\mathcal{H}}  \l \si(t_1),\ldots, \si(t_n) \r, \qquad \forall t_j \in \R_+,\ j \in \{1,\ldots, n\}.
\end{align*}
If a stable process is self-similar, then there are restrictions on the spectral measure. Consistently with Lemma 7.3.1 of \cite{samorodnitsky1994stable}, a typical form of the spectral measure is 
\begin{align}
    \Gamma _{\bm{t}_n} (d\bm{s}) = \sum _{j=1}^n t_j ^{\mathcal{H}\alpha}\, \nu _j  (d \bm{s}), \label{Misura spettrale nel caso di sigma autosimile con H generico}
\end{align}
for suitable signed measures $\nu _j$, $j \in \{1,\ldots, n\},\ n\in\N_+$.
Under this assumption, the vector $\bm{\si}(\bm{t}_n)$ has characteristic function
    \begin{align*}
        \E e^{i \langle \bm{\xi}, \bm{\si(\bm{t}_n)} \rangle} = \exp\lg -\sum_{j = 1}^n t_j^{\al \mathcal{H}} \Int_j(\bm{\xi})\rg, \qquad \bm{\xi} = (\xi_1,\ldots, \xi_n)\in \R^n,
    \end{align*}
    where $\Int_j(\bm{\xi}) := \int_{\Omega} \l -i \langle \bm{\xi}, \bm{s}\rangle  \r^{\al} \nu_j(d\bm{s}),\ j \in \{1,\ldots, n\},\ n\in\N_+$. Furthermore, we observe that for $k \in \{1,\ldots, n\}$ and $\xi_j = 0,\ j \neq k$ the previous equation must be the characteristic function of $\si(t_k)$ so that the dependence on $t_j,\ j\neq k$ must disappear on the right hand side. In other words, using $\bm{e}_j$ for the $j$th vector of the standard basis of $\R^n$, for $j \neq k$ one has
    \begin{align*}
        \int_{\Omega}
        \l -i \langle \xi_k \bm{e}_k, \bm{s} \rangle \r ^{\al}
         \nu_j(d\bm{s}) = 0
    \end{align*}
    which implies $ \Int_j(\xi_k \bm{e}_k) = 0$. As a consequence,
    \begin{align*}
        %\xi_k\bm{e}_k) = \l -i\,\xi_k\r^{\al} \delta_{jk}, \qquad \xi_k \in \R, \\
         \Int_j(\xi_k\bm{e}_k) = \delta_{jk} \int_{\Omega}\l -i\,\xi_k s_k\r^{\al} \nu_k(d\bm{s}), \qquad \xi_k \in \R, 
    \end{align*}
    must hold, where $\delta_{jk}$ indicates the Kronecker delta.

\begin{os}
\begin{comment}
    It is worth noting that a non-negative, non-decreasing stable process can be assumed to be self similar but still not Lévy. Indeed, using the notation of Remark \ref{Remark on stable non negative non decreasing process} it is sufficient to require the kernel $f_t(\cdot)$ to be an homogeneous function of order $H$ in the parameter $t$, so that
    \begin{align*}
        X(ct) = c^H\int_0^{\infty} f_t(x) dM_{\al}(x) = c^H X(t).
    \end{align*}
    \end{comment}
    
    For a construction of such a process, see Lemma 7.3.2 in \cite{samorodnitsky1994stable} with $\beta_{X(1)} = 1$ to ensure positivity and  $\mathcal{H} \neq 1/\al$ to ensure dependent increments. The drift parameter $\mu_{X(1)}$ can be chosen equal to 0 for the sake of simplicity.
\end{os}

\begin{os}
If $\sigma$ is a  $\al$-stable subordinator, then the signed measures $\nu _j$ are discrete. Indeed, considering the example $n=2$,  by using independence and stationarity of increments, the characteristic function of $\l \si(t_1), \si(t_2) \r$ has the following form
\begin{align*}
    \E e^{i \l \xi_1 \si(t_1)  + \xi_2 \si(t_2)\r} &= \E e^{i (\xi_1 + \xi_2) \si(t_1)} \E e^{i \xi_2 \si(t_2 - t_1)}  \\
    &= e^{- t_1 \l -i (\xi_1 + \xi_2) \r^{\al}} e^{- (t_2-t_1) (-i\xi_2)^{\al}} \\
     &= e^{-\int_{\Omega} \l -i(\xi_1 s_1 + \xi_2 s_2)\r^{\al} \lq t_1 \l 2^{\frac{\al}{2}} \delta_{\l 1/\sqrt{2}, 1 /\sqrt{2} \r}(ds_1,ds_2) - \delta_{\l 0, 1 \r}(ds_1,ds_2) \r  + t_2 \delta_{\l 0, 1 \r} (ds_1,ds_2)\rq}
\end{align*}
so that here
\begin{align*}
    \nu_1(ds_1,ds_2) = \l 2^{\frac{\al}{2}}\delta_{\l 1/\sqrt{2}, 1 /\sqrt{2} \r}(ds_1,ds_2) - \delta_{\l 0, 1 \r}(ds_1,ds_2) \r, \qquad \nu_2(ds_1,ds_2) = \delta_{\l 0, 1 \r}(ds_1,ds_2) .
\end{align*}
Here we used the notation $\delta_{(a,b)}$ for the Dirac delta in $ (a,b) \in \R^2$.
\end{os}

\subsubsection{Remarkable examples}
\label{remarkable cases of stable sigma}

This section is devoted to the analysis of some remarkable sub-classes of self-similar processes.

The first one concerns the case $\mathcal{H} = 1 / \al$. Here the spectral measure is linear in the time variables, so that Formula \eqref{Misura spettrale nel caso di sigma autosimile con H generico} becomes
\begin{align*}
    \Gamma _{\bm{t}_n} (d\bm{s}) = \sum _{j=1}^n t_j \, \nu _j  (d \bm{s})
\end{align*}
and a nice formula for the joint law of the waiting times follows, as shown in the  result below. The Lévy stable process is obtained by further assuming independence of increments.

The second is the case of the so-called $\mathcal{H}$-sssi processes, namely the $\mathcal{H}$-self-similar processes with stationary increments.
In this case (see  Section 7.3 of \cite{samorodnitsky1994stable}), since $\al \in (0,1)$ then the Hurst parameter must vary in $(0,1/\al]$.
%The value $H = 1/\alpha$ gives the Lévy stable process. 
We immediately use Corollary 7.3.4 in \cite{samorodnitsky1994stable} to conclude that there exists such a process which is also non negative, strictly increasing, but with dependent increments. To do this, using the notation of the book, it is sufficient to choose $\beta_{X(1)} = 1,$ and $\mathcal{H} \neq 1/\al$. Instead, for $\mathcal{H} = 1/\al$ the Lévy case is recovered. The parameter $\mu_{X(1)}$ can be chosen equal zero for the sake of simplicity.

We are ready to present a corollary of Theorem \ref{distribuzione intertempi}.

\begin{cor}
Consider the assumptions of Theorem \ref{distribuzione intertempi}. Furthermore, let $\sigma$ be an $\mathcal{H}$-self-similar $\al$-stable process. Then the following claims hold true:
\begin{enumerate}
\item If $\mathcal{H}=1 / \al$, then the waiting times have the following joint distribution
\begin{align} \label{congiunta finale}
     \mathbb{E} e^{i\langle\bm{\xi}, \bm{J}_n \rangle}= \prod _{k=1}^n \frac{\lambda}{\lambda + \sum _{j=k}^n \Int_j(A^T\bm{\xi})}, \qquad \bm{\xi} \in \R^n,
\end{align}
where $\Int_j(\bm{\xi}) = \int_{\Omega} \l -i \langle \bm{\xi}, \bm{s}\rangle  \r^{\al} \nu_j(d\bm{s})$.

\item If $\sigma$ is $\mathcal{H}$-sssi, ($\mathcal{H}\in ( 0, 1/\alpha]$), then all the waiting times have the same marginal distribution, given by
\begin{align*}
    \E e^{i \xi J_1 
    }= \lambda \int_0^{\infty} e^{- \l -i\xi \r^{\al} w^{\mathcal{H}\al} -\lambda w} dw.
\end{align*}

%\item If $H=1$ then the waiting times have a Schur constant distribution.

\end{enumerate}

\end{cor}

\begin{proof}
In the case of $\mathcal{H}$-self-similar processes, one has 
\begin{align*}
    \Gamma_{B \, \bm{w}_n}(ds) = \sum_{j = 1}^n \l \sum_{k = 1}^j w_k\r^{\mathcal{H}\alpha}\nu_j(ds).
\end{align*}

For $\mathcal{H} = 1/\al$, one also has 

\begin{align*}
    \int_{\Omega} \l -i \langle A_n^T \bm{\xi}, \bm{s}\rangle \r^{\al} \Gamma_{B \, \bm{w}_n}(ds) = \sum_{j = 1}^n \sum_{k = 1}^j w_k \Int_j\l A_n^T\bm{\xi} \r = \sum_{k = 1}^n \sum_{j = k}^n w_k \Int_j\l A_n^T\bm{\xi} \r. 
\end{align*}
Here the equation \eqref{funzione caratteristica intertempi nel caso stabile} becomes 
\begin{align*}
    \E e^{i \langle \bm{\xi}, \bm{J}_n \rangle} = \lambda^n \int_{\R_+^n} \prod_{k = 1}^n e^{-w_k \l \sum_{j = k}^n \Int_j\l A_n^T \bm{\xi} \r + \lambda\r}dw_1 \ldots dw_n.
\end{align*}
An application of Fubini theorem leads to \eqref{congiunta finale}.

To prove \textit{(2)}, we  observe that, because of the property of stationary increments, we have
\begin{align*}
     J_n=\si \l \sum _{k=1}^{n} W_k  \r - \si \l \sum _{k=1}^{n-1} W_k \r \overset{d}{=} \si(W_n),
\end{align*}
so that the waiting times marginally have the same distribution of $\si(W_1)$. Specifically, using Corollary 7.3.4 in \cite{samorodnitsky1994stable} we have $ \E e^{i \xi \si(t)} = e^{- \l -i\xi \r^{\al} t^{\mathcal{H}\al}} $ and a conditioning argument gives
\begin{align*}
    \E e^{i \xi \si(W_1)} = \int_0^{\infty} e^{- \l -i\xi \r^{\al} w^{\mathcal{H}\al}} \lambda e^{-\lambda w} dw, \qquad \xi\in\R.
\end{align*}

\end{proof}
\begin{os}
If $\si$ is a Lévy stable subordinator, the waiting times of the subordinated process are i.i.d.~Mittag-Leffler, namely Formula \ref{congiunta finale} reduces to
\begin{align} 
     \mathbb{E} e^{i\langle\bm{\xi}, \bm{J}_n \rangle}= \prod _{k=1}^n \frac{\lambda}{\lambda +(-i \xi_k )^\alpha}\qquad \bm{\xi} \in \R^n
\end{align}
and thus the chain $Y$ is semi-Markovian.
\end{os}

\vspace{1cm}

\section*{\small Acknowledgements}

All the authors acknowledge financial support under the
National Recovery and Resilience Plan (NRRP), Mission 4, Component 2, Invest-
ment 1.1, Call for tender No. 104 published on 2.2.2022 by the Italian Ministry
of University and Research (MUR), funded by the European Union – NextGenerationEU– Project Title “Non–Markovian Dynamics and Non-local Equations”
– 202277N5H9 - CUP: D53D23005670006 - Grant Assignment Decree No. 973
adopted on June 30, 2023, by the Italian Ministry of Ministry of University and
Research (MUR).

Lorenzo Facciaroni acknowledges financial support from the Sapienza University of Rome, Fondo di Avvio alla Ricerca 2025 (Grant No. AR125199BF489F9E).

\section*{\small
 Conflict of interest} %%%%%%%%%%%%%%%%%%%%

 {\small
 The authors declare that they have no conflicts of interest.}

\printbibliography

\clearpage

{\bf Lorenzo Facciaroni}

Department of Statistical Sciences, Sapienza - University of Rome

Email: \texttt{lorenzo.facciaroni@uniroma1.it}

\bigskip

{\bf Costantino Ricciuti}

Department of Statistical Sciences, Sapienza - University of Rome

Email: \texttt{costantino.ricciuti@uniroma1.it} 

\bigskip

{\bf Enrico Scalas}

Department of Statistical Sciences, Sapienza - University of Rome

Email: \texttt{enrico.scalas@uniroma1.it}

\end{document}